\documentclass[a4paper]{scrartcl}
\usepackage{amsmath}
\usepackage{amssymb}
\usepackage{amsfonts}
\usepackage{amsthm}
\usepackage{graphicx}
\usepackage{123}
\usepackage{rotating}
\usepackage[margin=3cm,centering]{geometry}
\theoremstyle{definition}
\newtheorem{definition}{Definition}[section]
\theoremstyle{plain}
\newtheorem{lemma}{Lemma}[section]
\newtheorem{theorem}{Theorem}[section]
\newtheorem{corollary}{Corollary}[section]
\theoremstyle{remark}

\newtheorem{proposition}{Proposition}[section]

\newcommand{\abs}[1]{\left\lvert#1\right\rvert}
\newcommand{\norm}[1]{\left\lVert#1\right\rVert}
\DeclareMathOperator{\id}{id}
\title{Geometry of Numerical Complex Time Integration}
\author{Thorsten Orendt\thanks{Zentrum Mathematik der Technischen Universit\"at M\"unchen, Boltzmannstr. 3, 85748 Garching bei M\"unchen, Germany (orendt@ma.tum.de, richter@ma.tum.de, schmid@ma.tum.de). Authors in alphabetical order.} \and J\"urgen Richter-Gebert$^{*}$ \and Michael Schmid$^{*}$}
\begin{document}
\maketitle
\begin{abstract}
We are studying \textsc{Runge-Kutta} methods along complex paths of integration from a geometric point of view. 
Thereby we derive special complex time grids, which applied to the problem of integrating a linear autonomous system of ordinary 
differential equations, can be used to achieve a classical superconvergence effect. The approach is also adapted for arbitrary ODEs. 
Furthermore we draw a connection from our geometric reasoning to the class of composition methods with complex coefficients.
Thereby, our main goal is to introduce a new point of view on these methods.
\end{abstract}
%
\vspace{1cm}
%
\section{Introduction}
Numerical integration of ODEs is in essence a well understood subject and there is a wide range
of methods that are applicable for different flavors of the problem. This paper 
aims at adding some novel twist to the subject: {\it superconvergence in the presence of certain
complex time grids and a new viewpoint on composition methods with complex coefficients.} Usually numerical integration of ODEs is embedded into a context where
the idealized time variable is considered to be a continuously flowing {\it real} entity. 
Numerical integration methods (like for instance variants  of the \textsc{Runge-Kutta} approach) 
approximate this continuous flow by time steps of finite resolution. 
Usually the finer the grid the smaller the global approximation error. 
Adaptive step width algorithms aim to produce a compromise of numerical accuracy and computational effort. 
Near numerically instable situations, adaptive stepwidth algorithms tend to introduce many additional grid points. 
In fact, in a scenario of time flow within the real numbers there is no way to bypass these numerically critical situations. 
For instance, imagine a version of the restricted two-body problem where a 
planet moves around a fixed star under Newtons law of gravity. If the initial velocity of the planet points directly towards the sun, 
the system will necessarily  run into a singular situation. Scenarios nearby this situation will cause numerically 
difficult situations. In a naive adaptive step width approach usually many additional time steps are introduced when the 
planet swings closely around the sun. A setup that allows also {\it complex} time, leaves the possibility to circumvent the singular 
situation by introducing a time flow that makes a ``detour'' through complex values.

Our research was initially motivated by the question of applying complex detours 
to resolve singularities in numerical integration for ODEs. The philosophy behind this 
approach has been successfully applied to other situations. For instance in the field of Dynamic Geometry 
(see for instance \cite{KRG2}) the second author successfully applied complex detours to avoid discontinuous behavior, 
which was known to be a notorious problem in this field \cite{KRG1}. In relation to making real-time numerical 
simulation available to the dynamic geometry program \textsc{Cinderella}, it was natural to study a complex detour approach, as well.

Investigating complex detours for ODEs raises several interesting research questions.
Besides many interesting implementation-related issues, in particular problems arise, which are related to
the analytic character of the input and the output of the solver.

\begin{itemize}
\item{\bf Analytic continuation of the right side:\enspace}
Functions in a complex scenario require the proper treatment of different branches (for instance for $\sqrt{\ldots}$
and $\log(\ldots)$).
To avoid path jumping (as mentioned for dynamic geometry) all function evaluations have to mimic an 
analytic behavior. Thus in the evaluations of intermediate results the ``correct'' branches have to be chosen. 
This requires the implementation of complex tracing strategies for the evaluation of the right side of the ODE if the derivative in the ODE has several branches.
\item{\bf Complex manifold of the solution:\enspace} 
Not only the right side of the ODE may exhibit monodromy behavior of an analytic function. Also the solution itself may
depend on the actual path chosen for the time variable. Even if the right side of the ODE is unbranched, it may happen that 
the solution at a concrete moment of time depends on the integration path, connecting the start time and the end time.
Such effects have for instance been studied in \cite{CGSS05},
where also a connection of chaotic behavior and these monodromy effects is drawn.
\end{itemize}

Perhaps due to these reasons (the additional implementation problems and the path dependence of the solution) 
in the literature there are only very few cases in which the paradigm of {\it real time} is broken \cite{BORNE}. 
Still we are convinced that the complex setup is worth to be studied both from a practical standpoint that asks for good 
approximations with small computational effort, as well as from a structural purely mathematical standpoint. 
In the first part of this article we will focus on a specific situation, which in a surprising way connects these two aspects. 
We will study linear ODEs, which are interesting since they form a perhaps simplest possible scenario,
in which one can apply complex detours. For a linear right side one has just one single branch, so we can neglect the analytic continuation issues. 
Surprisingly, in this case carefully chosen complex time grids can help to reduce the global integration
error by at least one order. Nevertheless, there is also a disadvantage of this approach. Roughly speaking, the effort
of complex computations is six times higher than in the real-valued setup. Since there exist 
integration methods of arbitrary order, it is natural to ask the following question:

\emph{Why should one study numerical integration along complex time grids to increase the order of a one-step method by one, if one could 
use a real-valued one-step method, which has already got the increased order?}

The answer is, that the authors have not been faced with this question at the beginning of their studies concerning techniques to detour singularities,
but for the reader, which is familiar with the field of numerical integration, it should be interesting that the naive and geometric approach of the authors
has led to a new derivation of the class of composition methods with complex coefficients, 
apart from the usual way of solving some suitable set of order conditions over the reals or 
the complex numbers like done in \cite{HLW02,HO08,V08}. In addition to that, the authors also discovered some results concerning the convergence of 
\textsc{Runge-Kutta} methods.

The paper is organized as follows. In Section 2 we present a simplest possible scenario where the superconvergence 
effects studied in this article arise. This example (namely complex detours for $\dot{x} = x$) 
will serve as a motivating paradigm for our further considerations. Section 3 introduces the necessary setup of complex time grids in relation to 
\textsc{Runge-Kutta} methods. Section 4 is the main technical part of this article. We first deal with the problem under which
condition a \textsc{Runge-Kutta} method applied to a complex path yields a real terminal point (Theorem \ref{cor:realValued}). After this we proof 
a main result of this article: One can increase the order of a convergence of a \textsc{Runge-Kutta} method applied to a linear ODE by choosing a 
suitable complex path (Corollary \ref{cor:superconvergence}). We gain a lot of geometrical insight in the structure of the path 
and can derive explicit criteria that have to be satisfied in order to obtain superconvergence. 
These criteria are closely related to the multiplicative structure of roots of unity in the complex plane. 
Section 5 draws the connection of the complex detours to composition methods, which enables us to extend our method (at least in a certain sense) 
also to the case of a non-linear right-hand side. 
The condition equation for superconvergence derived in Section 4 are closely related to the classical criteria for increasing the order
of composition methods. A composition method is obtained by applying a integration method (say a \textsc{Runge-Kutta} integration)
consecutively in a controlled way. The sequence of these applications can again be considered as {\it one} step of
a more complicated integration method. Under certain circumstances (the above mentioned criteria) one can increase the order of 
the original method by at least one. Our geometric approach allows to interpret some of these composition methods entirely on the level 
of an underlying integration grid. Thus these composition methods simply correspond to a suitably chosen complex detour. 
We can even iterate this process and by this obtain (Section 5.3) composition methods of arbitrary high degree (also known as the ``Yoshida trick''). 
Surprisingly, the corresponding paths that encode the composition structure exhibit a fractal structure. 
Finally at the end of Section 5, we illustrate our methods by the more sophisticated problem of computing the \textsc{Arenstorf} orbit.
\section{A motivating example}
In this section we would like to make the reader familiar with the subject of numerical integration along complex paths (and its benefits) in an
informal but hopefully self-explanatory way. In order to concentrate on the fundamental ideas, this and the follwing section is written in
a very easy and self-containing manner.
  
We choose the simplest possible initial value problem
\begin{equation} \label{eq:ivpexp}
	\dot{x}(t) = f\big(t, x(t)\big), \quad x(t_0) = x_0,
\end{equation}
where $f: \mathbb{R}^2 \to \mathbb{R}, ~ (x, t) \mapsto x$, and $(t_0, x_0) := (0, 1)$. The corresponding 
analytic solution curve $\varphi: \mathbb{R} \rightarrow \mathbb{R}, ~ t \mapsto \varphi(t)$, is given by the well-known exponential 
function $\varphi(t) := e^t, ~ \forall t \in \mathbb{R}$.

In order to compute $\varphi(1)$, a standard approach is given by using an explicit {\sc Runge-Kutta} method (eRKM) along an equidistant decomposition
of the interval $[0, 1]$.
%

\medskip

In contrast to the real interval $[0,1]$, we now study the use of complex paths connecting $0$ and $1$.
As (\ref{eq:ivpexp}) is a well-behaved\footnote{For existence and uniqueness of a complex solution curve have a look at \cite{SS73}.} 
initial value problem given by an autonomous linear first order differential equation 
with constant coefficients, $t \in \mathbb{R}$ seems to be an unnecessary restriction. To be more precisely we note the following. 

Our initial value problem (\ref{eq:ivpexp}) could be written equivalently as the integral equation
\begin{equation*}
	x(t) = x_0 + \int_{t_0}^t x(s) ds = x_0 + \int_\gamma x(z) dz,
\end{equation*}
where $\gamma: \mathbb[t_0, t] \to \mathbb{C}, ~ s \mapsto s$, is a $\mathcal{C}^1$-curve. 
As the solution 
 $\varphi$ has to be a holomorphic function\footnote{In the case of problem (\ref{eq:ivpexp}), 
there exists a unique entire function as solution curve for every choice of $(t_0, x_0) \in \mathbb{C}^2$ as initial condition.},
\textsc{Cauchy}'s well-known integral formula states, that $\varphi(t)$ is independent of the detailed choice of $\gamma$. 
Every $\mathcal{C}^1$-curve $\gamma$, with starting point $t_0$ and ending point $t$, yields the same result for $\varphi(t)$.

With this ``analytic'' picture in mind, we have a look at the behavior of the explicit \textsc{Euler} method along the following
complex time grid. We now take
\begin{equation*}
	\tilde{t}_j := \frac{1}{2} \left(e^{i \pi \left(1 - \frac{j}{n}\right)} + 1\right), \quad \forall j \in \{0, \ldots, n\},
\end{equation*}
as a discretization of the upper complex half circle from $0$ to $1$. 
Starting at the initial value $x_0 = 1$, the explicit \textsc{Euler} method generates $n$, not necessary real-valued, approximation values
\begin{equation*}
	\tilde{x}_j \approx \varphi(\tilde{t}_j) = e^{\tilde{t}_j}, \quad \forall j \in \{0, \ldots, n\}.
\end{equation*}
In the following, we compare the explicit \textsc{Euler} method along the mentioned complex time grid and the equidistant composition
of the interval $[0, 1]$, given by $t_j := \frac{j}{n}, ~ \forall j \in \{0, \ldots, n\}$.

Figure \ref{fig:complexInt} illustrates this construction in the real and complex case. Note that Figure \ref{fig:complexInt} represents 
only the image of the corresponding \textsc{Euler} polygons ($w$-plane). Thus the real-valued construction degenerates onto 
the real line.
\begin{figure}[hbt]
	\begin{center}
		\psset{xunit=4cm,yunit=4cm,runit=4cm}
		\begin{pspicture}(0.9,-0.1)(3,1)
			\psline(0.9,0)(3,0)
			\psline(1,-0.1)(1,1)
			\psdot[dotstyle=diamond,dotsize=0.04](1,0)
			\psdot[dotstyle=diamond,dotsize=0.04](1.1,0)
			\psdot[dotstyle=diamond,dotsize=0.04](1.21,0)
			\psdot[dotstyle=diamond,dotsize=0.04](1.331,0)
			\psdot[dotstyle=diamond,dotsize=0.04](1.4641,0)
			\psdot[dotstyle=diamond,dotsize=0.04](1.61051,0)
			\psdot[dotstyle=diamond,dotsize=0.04](1.771561,0)
			\psdot[dotstyle=diamond,dotsize=0.04](1.9487171,0)
			\psdot[dotstyle=diamond,dotsize=0.04](2.14358881,0)
			\psdot[dotstyle=diamond,dotsize=0.04](2.357947691,0)
			\psdot[dotstyle=diamond,dotsize=0.05](2.59374246,0)
			\psdot[dotstyle=square,dotsize=0.04](1,0)
			\psdot[dotstyle=square,dotsize=0.04](1.0244,0.1545)
			\psdot[dotstyle=square,dotsize=0.04](1.075693,0.308276)
			\psdot[dotstyle=square,dotsize=0.04](1.16058,0.461)
			\psdot[dotstyle=square,dotsize=0.04](1.289,0.608)
			\psdot[dotstyle=square,dotsize=0.04](1.4739,0.733)
			\psdot[dotstyle=square,dotsize=0.04](1.719,0.8108)
			\psdot[dotstyle=square,dotsize=0.04](2.016,0.801)
			\psdot[dotstyle=square,dotsize=0.04](2.3287,0.6673)
			\psdot[dotstyle=square,dotsize=0.04](2.5871,0.3901)
			\psdot[dotstyle=square,dotsize=0.05](2.7107,0.0)
			\psdot[dotstyle=x,dotsize=0.04](1,0)
			\psdot[dotstyle=x,dotsize=0.04](1.1052,0)
			\psdot[dotstyle=x,dotsize=0.04](1.2214,0)
			\psdot[dotstyle=x,dotsize=0.04](1.3499,0)
			\psdot[dotstyle=x,dotsize=0.04](1.4981,0)
			\psdot[dotstyle=x,dotsize=0.04](1.6487,0)
			\psdot[dotstyle=x,dotsize=0.04](1.8221,0)
			\psdot[dotstyle=x,dotsize=0.04](2.0138,0)
			\psdot[dotstyle=x,dotsize=0.04](2.2255,0)
			\psdot[dotstyle=x,dotsize=0.04](2.4596,0)
			\psdot[dotstyle=x,dotsize=0.05](2.7183,0)
			\psdot[dotstyle=+,dotsize=0.04](1,0)
			\psdot[dotstyle=+,dotsize=0.04](1.0126,0.1577)
			\psdot[dotstyle=+,dotsize=0.04](1.0530,0.3187)
			\psdot[dotstyle=+,dotsize=0.04](1.1297,0.4836)
			\psdot[dotstyle=+,dotsize=0.04](1.2559,0.6467)
			\psdot[dotstyle=+,dotsize=0.04](1.4469,0.7904)
			\psdot[dotstyle=+,dotsize=0.04](1.7107,0.8809)
			\psdot[dotstyle=+,dotsize=0.04](2.0335,0.8706)
			\psdot[dotstyle=+,dotsize=0.04](2.3648,0.7157)
			\psdot[dotstyle=+,dotsize=0.04](2.6210,0.4082)
			\psdot[dotstyle=+,dotsize=0.05](2.7183,0.0)
		\end{pspicture}
		\caption{	Real and complex approximation for n = 10. 
							Thereby $+$ and $\times$ represent the corresponding exact values of the 
							solution $\varphi$ along the upper complex half circle respectively the real interval from 0 to 1.
							{\tiny$\square$} and $\diamond$ are the corresponding values of the explicit \textsc{Euler} method.}
		 \label{fig:complexInt}
	\end{center}
\end{figure}
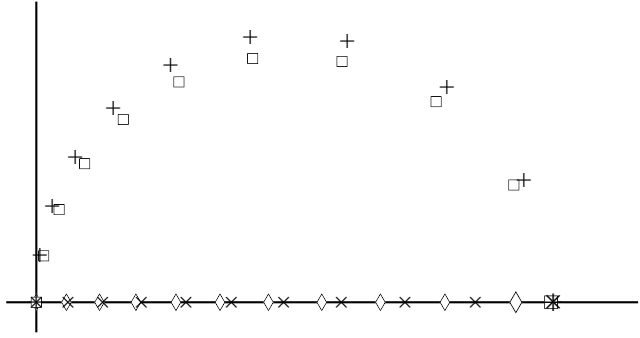
\noindent
Looking at the numerical results,
\begin{align*}
	\renewcommand{\arraystretch}{1.3}
	\begin{array}{c|c|c}
		\text{step} ~ j & x_j ~ (\diamond) & \tilde{x}_j ~ ({\text{\tiny$\square$}}) \\
		\hline
		0 & 1 & 1 \\
		1 & 1.100000000 & 1.024471742 + 0.1545084969 \cdot i \\
		2 & 1.210000000 & 1.075693448 + 0.3082767551 \cdot i \\
		3 & 1.331000000 &	1.160581914 + 0.4613658247 \cdot i \\
		4 & 1.464100000 & 1.289582523 + 0.6080971485 \cdot i \\
		5 & 1.610510000 & 1.473952784 + 0.7336116556 \cdot i \\
		6 & 1.771561000 & 1.719643769 + 0.8108906981 \cdot i \\
		7 & 1.948717100 & 2.016924082 + 0.8017873023 \cdot i \\
		8 & 2.143588810 & 2.328718297 + 0.6673738888 \cdot i \\
		9 & 2.357947691 & 2.587124642 + 0.3901842509 \cdot i \\
		10 & 2.593742460 & 2.710722870 - 0.0000000006 \cdot i
	\end{array}
\end{align*}
we point out the following observations:
\begin{enumerate} 
	\item	Both terminal points seem to be an approximation of $e \approx 2.718281828$ (as expected).
	\item	The imaginary part of the complex terminal point seems to be zero. \label{enum:observations}
	\item	The complex construction yields more correct digits! To be more precise, the complex terminal point $\tilde{x}_{10}$ 
			is more than $16$ times closer to $e$ than the real terminal point $x_{10}$.
\end{enumerate}
Especially the last observation seems a worthwhile phenomenon to go into a deeper study of problem (\ref{eq:ivpexp}) and
related ones. In this article we will develop a theorem\footnote{Compare Theorem \ref{thm:mainTheorem} and Corollary \ref{cor:superconvergence}.}, 
which explains this phenomenon in the more general context of numerical integration of linear initial value problems with
constant coefficients by the use of explicit \textsc{Runge-Kutta} methods (eRKM)s along complex paths.
\section{Complex flows and one-step methods}
Before we confine our attention to the encountered observations above, we have to fix the notation in this section.
\begin{definition}[CIVP] \label{definitionCIVP}
	Let $d \in \mathbb{N}, ~ \Omega_f \subseteq \mathbb{C} \times \mathbb{C}^d$ and 
	$f: \Omega_f \to \mathbb{C}^d, ~ (t, x) \mapsto f(t, x)$, be 
	a continuous function. Then
	\begin{equation}
		\dot{x}(t) = f \big( t, x(t) \big), \label{ode}
	\end{equation}
	is called an \emph{explicit first order differential equation}. A \emph{solution curve} of (\ref{ode}) is a complex differentiable function
	$\varphi : U \to \mathbb{C}^d, ~ t \mapsto \varphi(t)$, with
	\begin{enumerate}
		\item	$\Gamma_\varphi := \left\{ \big( t, \varphi(t) \big) \big| t \in U \right\} \subseteq \Omega_f$ and
		\item	$\dot{\varphi}(t) =  f \big( t, \varphi(t) \big), ~ \forall t \in U$,
	\end{enumerate}
	where $U \subseteq \mathbb{C}$ is an open set. (\ref{ode}) together with a point $(t_0, x_0) \in \Omega_f$, the \emph{initial condition}, 
	is called a \emph{complex initial value problem (CIVP)}. A \emph{solution} to a (CIVP) is a solution curve
	$\varphi: U \rightarrow \mathbb{C}^d$ of (\ref{ode}), where $t_0 \in U$ and
	\begin{equation*}
		\varphi(t_0) = x_0.
	\end{equation*}
\end{definition}
\subsection{Complex flow}
Given a path $\gamma$, with $\gamma(0) = t_0 \in \mathbb{C}$ and $\gamma(1) = t_1 \in \mathbb{C}$. Let us assume that
for every $y \in N$, where $N$ is a neighborhood of $x_0 \in \mathbb{C}^d$, there exists a local solution $\varphi_{(t_0, y_0)}$ to the (CIVP)
\begin{equation*}
	\dot{x}(t) = f\big(t, x(t)\big), \quad x(t_0) = y_0 \in \mathbb{C}^d.
\end{equation*}
For every such solution $\varphi_{(t_0, y_0)}$ let us assume, that the analytic continuation of $\varphi_{(t_0, y_0)}$ along $\gamma$,
denoted by $\tilde{\varphi}_{(t_0, y_0)}$, exists. In this situation, we have a map
\begin{equation*}
	\Phi_f^\gamma: N \rightarrow \mathbb{C}^d, ~ y_0 \mapsto \Phi_f^\gamma y_0 := \tilde{\varphi}_{(t_0, y_0)}(t_1).
\end{equation*}
\paragraph{Note}
As $\Phi_f^\gamma$ maps the initial value $y_0$ at $t_0$ to the corresponding 
``solution value'' $\Phi_f^\gamma y_0$ at $t_1$, we call this map the \emph{complex flow} (or the \emph{evolution induced by f}).
The complex flow can be interpreted as the path-dependent function, mapping an initial value $y_0$ to the value of the
analytic continuation $\tilde{\varphi}_{(t_0, y_0)}$ evaluated at $t_1$.
\subsection{Time grids and discrete flow}
In review of the introductory example, our presented real and complex approach took both use of 
discretizations of a given path $\gamma: [0, 1] \rightarrow U \subseteq \mathbb{C}$. 
Since $\mathbb{C}$ is not equipped with an order relation, we adapt the real-valued
concept of consecutive points in time $t_0 \leq \ldots \leq t_n \in \mathbb{R}$, by the use of indexed sets
\begin{equation*}
	\Delta := \left[ t^\Delta_0, \ldots, t^\Delta_{n_\Delta} \right] \subset \mathbb{C},
\end{equation*}
the \emph{time grids} with $n_\Delta \in \mathbb{N}$ \emph{time steps} $\tau^\Delta_j := t^\Delta_{j + 1} - t^\Delta_j, ~ j \in \{0, \ldots, n_\Delta - 1\}$. 
For every time grid
$\Delta$, we denote
\begin{equation*}
	\tau_\Delta := \max_{j \in \{0, \ldots, n_\Delta - 1\}} \abs{\tau^\Delta_j}
\end{equation*}
as the \emph{maximum step size} of $\Delta$. To increase readability, we will often omit the symbol $\Delta$ (that is, we will identify
$\tau_j = \tau_j^\Delta, ~ n = n_\Delta$, etc.), if mathematical definiteness does not suffer.

In the example of Section 2 our attempt was to construct a corresponding \emph{grid function}
\begin{equation*}
	x_\Delta : \Delta \rightarrow \mathbb{C}^d, ~ t \mapsto x_\Delta(t),
\end{equation*}
with
\begin{equation*}
	x_\Delta(t) \approx \varphi(t), \quad \forall t \in \Delta,
\end{equation*}
where $\varphi: U \rightarrow \mathbb{C}^d$ is a solution of (\ref{eq:ivpexp}) and $\Delta \subseteq U$.
To do so, we have started by setting $x_\Delta(t_0) := x_0$. Next, we have used \textsc{Euler}'s idea of small ``tangential'' update steps to recursively
compute the missing $x_\Delta(t_j)$'s. To generalize this idea, we just have to replace the method
for computing a new value $x_\Delta(t_{j + 1})$ from an already computed point $x_\Delta(t_{j})$. For this purpose, we introduce a function
\begin{equation*}
	\Psi : V \rightarrow \mathbb{C}^d, ~ (t, s, x) \mapsto \Psi^{t, s} x,
\end{equation*}
where $V \subseteq \mathbb{C} \times \mathbb{C} \times \mathbb{C}^d$ is an adequate set.
The variables $s$, $t$ and $x$ play the role of {\it current time}, {\it next time} and {\it current value}, respectively.    
At this point, one naturally assumes that $\Psi^{t, s} x$ is defined for every choice of 
$(t, s, x) \in \mathbb{C} \times \mathbb{C} \times \mathbb{C}^d$, with $(s, x) \in \Omega_f$ and $\abs{t - s}$ being sufficiently small. 
Under these assumptions $\Psi$ is often denoted as the \emph{discrete evolution} in analogy to the complex flow. 
To put things together, we define
\begin{equation*}
	x_\Delta(t_{j + 1}) := \Psi^{t_{j + 1}, t_j} x_\Delta(t_j), \quad \forall j \in \{0, \ldots, n_\Delta - 1\}.
\end{equation*}
For latter purposes, let
\begin{equation*}
	\Psi^\Delta x_0 := x_\Delta(t_{n}),
\end{equation*}
in analogy to $\Phi^\gamma x_0$. This is the terminal point reached when the discrete evolution $\Psi$ is developed along the time grid $\Delta$.
\subsection{Errors}
Given a time grid $\Delta := \left[t_0, \ldots, t_{n_\Delta}\right] \subset \mathbb{C}$, let $x_\Delta$ be a grid function generated by 
a one-step method applied to (\ref{ode}). As we have already seen by the example in the previous section, the approximation process is expected to cause \emph{grid errors} 
$\varepsilon_j := \varepsilon_\Delta(t_j)$, where $j \in \{0, \ldots, n_\Delta\}$ and
\begin{equation*}
	\varepsilon_\Delta : \Delta \rightarrow \mathbb{C}^d, ~ t_j \mapsto \Phi^{t_j, t_0} x_0 - x_\Delta(t_j).
\end{equation*}
At this point, we define for all $0 \leq i \leq j \leq n_\Delta$, $\Phi^{t_j, t_i} := \Phi^{\gamma^{j,i}_\Delta}$, where $\gamma^{j,i}_\Delta$ 
represents the traverse sequentially visiting $t_i, \ldots, t_j$.

The grid errors are generated by the local inaccuracy - the \emph{consistence error} - of the discrete evolution $\Psi$, denoted 
by\footnote{As $\Psi$ is a discrete evolution, $\varepsilon(t, x, \tau)$ is defined for all $(t, x) \in \Omega_f$, if $\tau$ becomes
sufficiently small.}
\begin{equation*}
	\varepsilon(t, x, \tau) := \Phi^{t + \tau, t} x - \Psi^{t + \tau, t} x.
\end{equation*}
Moreover, these local errors interact with the sensitivity of $\Phi^{t_j, t_0}$ to perturbations of $x$. In general they can be damped, 
amplified or fortunately be extinguished throughout the one-step recursion. A special variant of the latter case will be analyzed in what follows.
\section{Theorem linear case}
In this section we will reveal the secrets of the introductory example above in the more general context given by the (CIVP)
\begin{equation} \label{lCIVP}
	\dot{x}(t) = A x, \quad x(t_0) = x_0,
\end{equation}

where $A \in \mathbb{C}^{d \times d}, ~ d \in \mathbb{N}$ and $(t_0, x_0) \in \mathbb{C} \times \mathbb{C}^d$. 
Motivated by several numerical case studies, the effort to construct the corresponding \emph{complex} flow $\Phi^{t, t_0} x_0$ 
for an arbitrary choice of $t \in \mathbb{C}$ up to a desired accuracy by using a \textsc{Runge-Kutta} method (RKM) of order $p \in \mathbb{N}$, 
seems to be heavily dependent on the detailed choice of $\gamma$ - the path of integration. Pay attention to the fact, 
that there is no path-dependence from the analytical point of view, as already mentioned.

For every $n \in \mathbb{N}$, let
\begin{equation*}
	\Delta^\gamma_n := \left[ \gamma(0), \gamma\left(\frac{1}{n}\right), \ldots, \gamma\left(\frac{n - 1}{n}\right), \gamma(1) \right].
\end{equation*}
In the context of (\ref{lCIVP}), the convergence theory of one-step methods assures, that the family of grid functions 
$\left( x_{\Delta^\gamma_n} \right)_{n \in \mathbb{N}}$ corresponding to the used (RKM)
converges towards $\varphi(t) := \Phi^{t, t_0} x_0$ with order $p$ for an arbitrary choice of $\gamma$, where $\gamma(0) = t_0$ and $\gamma(1) = t$.
In the following we will show, that for every (RKM) of order $p$, there exists a path of integration $\gamma^*$, such that
\begin{equation*}
	\varepsilon_n^* := \varepsilon_{\Delta^{\gamma^*}_n} (t) = \mathcal{O}\left(\frac{1}{n^{p + 1}}\right).
\end{equation*}
Roughly speaking, the latter equation states, that there exists a path of integration $\gamma^*$, with $\gamma^*(0) = t_0$ and $\gamma^*(1) = t$, 
such that the error at the ending point is of order $p + 1$ - a superconvergence effect.\footnote{Note that the maximum grid error has not to be of 
order $p + 1$.} Furthermore, $\gamma^*$ can be chosen in such a way, that $\abs{\gamma^*} \leq \frac{\abs{t - t_0} \pi}{2}$, which means, 
that $\gamma^*$ is also of practical interest.
\subsection{Effects of complex conjugation}
Before we start to proof the main theorem of this section, sketched above, we take care of the second observation\footnote{Compare page \pageref{enum:observations}.} mentioned by our introductory example:

For a certain class of complex time grids, one can ensure that the computed value at the ending point of the time grid is real-valued.
Let us therefore have a look at the detailed structure of the discrete evolution $\Psi^{t + \tau, t} x$ of 
an arbitrary $s$-stage (eRKM) 
applied to problem (\ref{lCIVP}). If $\mathfrak{A}_{i,j}, \mathfrak{b}_i \in \mathbb{C}, ~ i, j \in \{1, \ldots, s\}$,
are the coefficients of the (eRKM), it follows that
\begin{equation*}
	\Psi^{t + \tau, t} x = x + \sum_{i = 1}^s \mathfrak{b_i} k_i,
\end{equation*}
where
\begin{equation*}
	k_i := A x + \tau \sum_{j = 1}^{i - 1} \mathfrak{A}_{i,j} A k_j
\end{equation*}
is the \emph{$i$-th stage} of the used (eRKM), $i \in \{1, \ldots, s\}$.
In fact, no matter what the coefficients of the (eRKM) are, if applied to the linear ODE (\ref{lCIVP}), then the evolution of one time step
$\tau \in \mathbb{C}$ from an initial point $x \in \mathbb{C}^d$
can be expressed as a simple matrix multiplication $M x$. Furthermore the matrix $M$ can be expressed
as a (matrix) polynomial in $\tau A$. The precise statement is captured by the following lemma.

\begin{lemma} \label{lemma:psiStructure}
	Let $s \in \mathbb{N}$ and $\Psi^{t + \tau, t} x$ be the discrete evolution of an arbitrary $s$-stage (eRKM) given by the coefficients 
	$\mathfrak{A} \in \mathbb{C}^{s \times s}$ and $\mathfrak{b}, \mathfrak{c} \in \mathbb{C}^s$ applied to problem (\ref{lCIVP}).
	Then it holds that
	\begin{equation*}
		\Psi^{t + \tau, t} x = P(\tau A) x, \quad \forall x \in \mathbb{C}^d,
	\end{equation*}
	where $P$ is a polynomial (applied to the matrix $\tau A$) of degree $s$, with coefficients
	\begin{equation*}
		p_0, \ldots, p_s \in \mathbb{Q}[\mathfrak{A}_{2,1}, \mathfrak{A}_{3,1}, \mathfrak{A}_{3,2}, \ldots, \mathfrak{A}_{s,1}, \ldots, \mathfrak{A}_{s, s - 1}, 
			\mathfrak{b}_1, \ldots, \mathfrak{b}_s].
	\end{equation*}
\end{lemma}
\begin{proof}
	Let us first have a look at the stages $k_i$. By induction over $i \in \{1, \ldots, s\}$, we will show that
	\begin{equation*}
		k_i = P_i(\tau A) A x,
	\end{equation*}
	where $P_i$ is a polynomial of degree $i - 1$ with coefficients in 
	$\mathbb{Q}[\mathfrak{A}_{2,1}, \ldots, \mathfrak{A}_{i,1}, \ldots, \mathfrak{A}_{i, i - 1}]$:
	\noindent
	Considering the first stage of the (eRKM) we get
	\begin{equation*}
		k_1 = A x = P_1(\tau A) A x,
	\end{equation*}
	where $P_1 := 1$. Next, let $i \in \{2, \ldots, s\}$. By using the induction hypothesis, it follows that 
	\begin{align*}
		k_i & = A x + \tau \sum_{j = 1}^{i - 1} \mathfrak{A}_{i,j} A k_j \\
		& = A x + \tau \sum_{j = 1}^{i - 1} \mathfrak{A}_{i,j} A P_j(\tau A) A x \\
		& = \left(I + \sum_{j = 1}^{i - 1} \mathfrak{A}_{i,j} \tau A P_j(\tau A)\right) A x.
	\end{align*}
	As one can easily see, the expression in the bracket above is induced by a polynomial $P_i$ satisfying the desired conditions.
	In conclusion,
	\begin{equation*}
		\Psi^{t + \tau, t} x = x + \tau \sum_{i = 1}^s \mathfrak{b}_i k_i = \left(I + \sum_{i = 1}^s \mathfrak{b}_i P_i(\tau A) \tau A \right) x
 	\end{equation*}
	yields the desired claim. 
	At this point we would like to make the reader aware of the fact, that there is no problem of commutativity in the matrix polynomials above, because of 
	$A^n A^m = A^m A^n$, for all $n, m \in \mathbb{N}_0$. 
\end{proof}
\paragraph{Remark} The polynomial $P$ of the previous lemma is also known as the \emph{stability function} of the underlying (eRKM).
\begin{definition}[symmetric time grid] \label{def:symmetricTimeGrid}
	Let $\Delta := \left[t_0^\Delta, \ldots, t_{n_\Delta}^\Delta\right]$ be a time grid. $\Delta$ is called \emph{symmetric}, if there exists a
	permutation $\pi \in \mathcal{S}_{n_\Delta}$ with $\pi^2 = \id$, such that
	\begin{equation*}
		\tau_j^\Delta = \overline{\tau_{\pi(j + 1) - 1}^\Delta},
	\end{equation*}
	for all $j \in \{0, \ldots, n_\Delta - 1\}$.
\end{definition}
\begin{figure}[hbt]
	\begin{center}
		\psset{xunit=0.8cm,yunit=0.8cm,runit=0.8cm}
		\begin{pspicture}(0,-5)(8,6)
			\put(1.5,5.3){$\mathbb{C}$}
			\psline[](0,1)(8,1)
			\psline[](1,-5)(1,6)
			\dotnode[linecolor=black](1,1){t0}
			\nput{135}{t0}{${}_{t_0^\Delta}$}
			\dotnode[linecolor=black](3,2){t1}
			\nput{0}{t1}{${}_{t_1^\Delta}$}
			\dotnode[linecolor=black](2,3){t2}	
			\nput{180}{t2}{${}_{t_2^\Delta}$}
			\dotnode[linecolor=black](3.5,5){t3}
			\nput{135}{t3}{${}_{t_3^\Delta}$}			
			\dotnode[linecolor=black](4.5,5){t4}
			\nput{45}{t4}{${}_{t_4^\Delta}$}
			\dotnode[linecolor=black](6,3){t5}
			\nput{0}{t5}{${}_{t_5^\Delta}$}
			\dotnode[linecolor=black](5,2){t6}
			\nput{180}{t6}{${}_{t_6^\Delta}$}
			\dotnode[linecolor=black](7,1){t7}
			\nput{45}{t7}{${}_{t_7^\Delta}$}
			\ncline[linecolor=black]{t0}{t1}
			\put(1.6,1.755555){${}_{\tau_0^\Delta}$}
			\ncline[linecolor=black]{t1}{t2}
			\put(2.05,2.33){${}_{\tau_1^\Delta}$}
			\ncline[linecolor=black]{t2}{t3}
			\put(2.8,3.85){${}_{\tau_2^\Delta}$}
			\ncline[linecolor=black]{t3}{t4}
			\put(3.85,4.8){${}_{\tau_3^\Delta}$}
			\ncline[linecolor=black]{t4}{t5}
			\put(4.8,3.85){${}_{\tau_4^\Delta}$}
			\ncline[linecolor=black]{t5}{t6}
			\put(5.5,2.35){${}_{\tau_5^\Delta}$}
			\ncline[linecolor=black]{t6}{t7}
			\put(6.1,1.7){${}_{\tau_6^\Delta}$}
			\put(5.8,4.1){$\gamma$}
			\pscurve[linestyle=dotted,linewidth=1pt]{}(1,1)(3,2)(2,3)(3.5,5)(4.5,5)(6,3)(5,2)(7,1)
			\dotnode(1,1){t0b}
			\nput{225}{t0b}{${}_{\tilde{t}_0^\Delta}$}
			\dotnode[linecolor=black](2,-1){t1b}
			\nput{-90}{t1b}{${}_{\tilde{t}_1^\Delta}$}
			\dotnode[linecolor=black](3.25,-1){t2b}	
			\nput{-90}{t2b}{${}_{\tilde{t}_2^\Delta}$}
			\dotnode[linecolor=black](4.5,-1){t3b}
			\nput{45}{t3b}{${}_{\tilde{t}_3^\Delta}$}			
			\dotnode[linecolor=black](5.25,-4){t4b}
			\nput{-90}{t4b}{${}_{\tilde{t}_4^\Delta}$}
			\dotnode[linecolor=black](6,-1){t5b}
			\nput{0}{t5b}{${}_{\tilde{t}_5^\Delta}$}
			\dotnode[linecolor=black](7,1){t6b}
			\nput{-45}{t6b}{${}_{\tilde{t}_6^\Delta}$}
			\ncline[linecolor=black]{t0b}{t1b}
			\put(1.65,0){${}_{\sigma_0^\Delta}$}
			\ncline[linecolor=black]{t1b}{t2b}
			\put(2.5,-0.65){${}_{\sigma_1^\Delta}$}
			\ncline[linecolor=black]{t2b}{t3b}
			\put(3.75,-0.55){${}_{\sigma_2^\Delta}$}
			\ncline[linecolor=black]{t3b}{t4b}
			\put(4.2,-2.4){${}_{\sigma_3^\Delta}$}
			\ncline[linecolor=black]{t4b}{t5b}
			\put(5.1,-2.4){${}_{\sigma_4^\Delta}$}
			\ncline[linecolor=black]{t5b}{t6b}
			\put(6.65,0){${}_{\sigma_5^\Delta}$}
			\put(5.95,-3.5){$\tilde{\gamma}$}
			\pscurve[linestyle=dashed,linewidth=1pt]{}(1,1)(2,-1)(3.25,-1)(4.5,-1)(5.25,-4)(6,-1)(7,1)
		\end{pspicture}
		\caption{	Two symmetric time grids $\Delta^{\gamma}_7 = \left[t_0^\Delta, \ldots, t_7^\Delta\right]$ 
					and $\Delta^{\tilde{\gamma}}_6 = \left[\tilde{t}_0^\Delta, \ldots, \tilde{t}_6^\Delta\right]$ along 
					$\gamma$ respectively $\tilde{\gamma}$.}
		 \label{fig:symmetricTimeGrid}
	\end{center}
\end{figure}
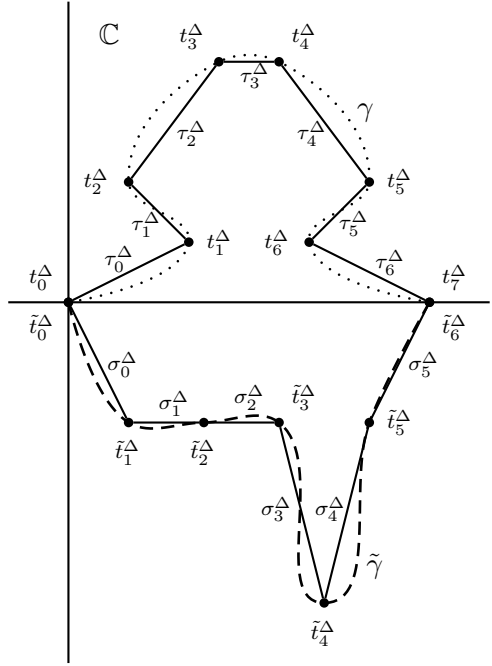
\begin{theorem} \label{cor:realValued}
	Let $\Delta_n^\gamma$ be a symmetric time grid with $t_0^{\Delta_n^\gamma} = t_0 \in \mathbb{R}$ and $t_n^{\Delta_n^\gamma} \in \mathbb{R}$, 
	where $n := n_\Delta$. 
	If $x_n := x_{\Delta_n^\gamma}\left(t_n^{\Delta_n^\gamma}\right)$ is constructed by using an $s$-stage (eRKM) with corresponding coefficients
	$\mathfrak{A} \in \mathbb{R}^{s \times s}$ and $\mathfrak{b}, \mathfrak{c} \in \mathbb{R}^s$ applied to problem (\ref{lCIVP}) 
	along $\Delta_n^\gamma$, where $A \in \mathbb{R}^{d \times d}$, then
	\begin{equation*}
		x_0 \in \mathbb{R}^d \quad \Rightarrow \quad x_n \in \mathbb{R}^d.
	\end{equation*}
\end{theorem}
\begin{proof}
	Let $\tau_j := \tau_j^{\Delta_n^\gamma}$, for $j \in \{0, \ldots, n - 1\}$. By using Lemma \ref{lemma:psiStructure}, it holds that
	\begin{align*}
		x_n = \prod_{j = 0}^{n - 1} P\left(\tau_j A\right) \cdot x_0,
	\end{align*}
	where $P$ is a polynomial of degree $s$ with real-valued coefficients. Therefore it is sufficient to show, that the product is in $\mathbb{R}^{s \times s}$.
 First we observe, that $P(\overline{\tau} A) = \overline{P(\tau A)}$ and $P(\tau A) P(\sigma A) = P(\sigma A) P(\tau A)$, for all $\sigma, \tau \in \mathbb{C}$.
	Thus,
	\begin{align*}
		\prod_{j = 0}^{n - 1} P\left(\tau_j A\right) & = 
			\prod_{\substack{j \in \{0, \ldots, n - 1\} \\ j + 1 = \pi(j + 1)}} \overbrace{P\left(\tau_j A\right)}^{\in \mathbb{R}^{s \times s}} \cdot
				\prod_{\substack{j \in \{0, \ldots, n - 1\} \\ j + 1 < \pi(j + 1)}} P\left(\tau_j A\right) P\left(\tau_{\pi(j + 1) - 1} A\right) \\
		& = \prod_{\substack{j \in \{0, \ldots, n - 1\} \\ j + 1 = \pi(j + 1)}} \underbrace{P\left(\tau_j A\right)}_{\in \mathbb{R}^{s \times s}} \cdot
			\prod_{\substack{j \in \{0, \ldots, n - 1\} \\ j + 1 < \pi(j + 1)}} 
				\underbrace{P\left(\tau_j A\right) \overline{P\left(\tau_j A\right)}}_{\in \mathbb{R}^{s \times s}},
	\end{align*}
	where $\pi \in \mathcal{S}_n$ is the corresponding permutation to the symmetric time grid $\Delta_n^\gamma$, with $\pi^2 = \id$. 
	Pay attention to the fact, that for every $M \in \mathbb{C}^{s \times s}$, it holds that 
	$M \overline{M} = \overline{M} M \Rightarrow M \overline{M} \in \mathbb{R}^{s \times s}$.
\end{proof}

As a result, Theorem \ref{cor:realValued} justifies the second observation of the introductory example in the more general context of a (CIVP) given by
an explicit linear autonomous system of first order differential equations. There the time grid was chosen to be symmetric along 
a half circle of the complex plane. Theorem  \ref{cor:realValued} states that if we start from a real point the terminal point will be real as well. 

\subsection{Main theorem linear case}
This section is dedicated to the third observation (accuracy gain at the ending point) of the introductory example. 
Let $\Delta$ be a time grid with starting point $t_0 \in \mathbb{C}$ and ending point $t \in \mathbb{C}$.
The main idea is, to ascribe the leading term of the error $\varepsilon_n = \varepsilon_\Delta(t)$ at the ending point to the step sizes $\tau^\Delta_j$.
This is achieved by a direct (but hard-core) expansion of $\Phi^{t, t_0} x_0 - \Psi^\Delta x_0$.
By doing so, we derive a condition for the step sizes $\tau^\Delta_j$, which enables us to increase the order of convergence at the ending point $t$.
\begin{theorem}[main theorem linear case] \label{thm:mainTheorem}
	Let $\gamma$ be a curve, with $\gamma(0) = t_0 \in \mathbb{C}$ and $\gamma(1) = t \in \mathbb{C}$. 
	For every $n \in \mathbb{N}$, we define $\delta_n := \tau_{\Delta^\gamma_n}$ and
	\begin{equation*}
		\tau_j(n) := t^{\Delta^\gamma_n}_{j + 1} - t^{\Delta^\gamma_n}_j, \quad j \in \{0, \ldots, n - 1\}.
	\end{equation*}
	Furthermore let $C \in \mathbb{R}^+$ be a constant, such that $\delta_n \leq C \frac{L(\gamma)}{n}$,
	where $L(\gamma) := \int_0^1 \abs{\dot{\gamma}(t)} dt$ is the \emph{length} of $\gamma$.
	
	If $\varepsilon_n$ denotes the grid error at the ending point $t \in \mathbb{C}$ corresponding to a $p$-stage (eRKM) of order 
	$p \in \mathbb{N}$ applied to problem (\ref{lCIVP}) along $\Delta_n^\gamma$, then 
	\begin{equation} \label{eqn:errorExpansion}
		\lim_{n \to \infty} \frac{\varepsilon_n}{\delta_n^p} = 
			\lim_{n \to \infty} \frac{1}{\delta_n^p} \sum_{j = 0}^{n - 1} \big( \tau_j(n) \big)^{p + 1} \cdot \frac{A^{p + 1} e^{(t - t_0) A}}{(p + 1)!} x_0.
	\end{equation}
\end{theorem}
\begin{proof}
	To increase readability we define for all $n \in \mathbb{N}$,
	\begin{equation*}
		\psi_j(n) := \sum_{k = 0}^p \frac{\big( \tau_j(n) A \big)^k}{k!} \quad \text{and} \quad 
			\omega_j(n) := \sum_{k = p + 1}^\infty \frac{\big( \tau_j(n) A \big)^k}{k!},
	\end{equation*}
	where $j \in \{0, \ldots, n - 1\}$. Furthermore let $E := e^{C L(\gamma) \norm{A}}$ and $x_n := \Psi^\Delta x_0$. To begin with our reasoning, we have a look at
	\begin{align*}
		\frac{\varepsilon_n}{\delta_n^p} & = \frac{\Phi^{t, t_0} x_0 - x_n}{\delta_n^p} = \frac{1}{\delta_n^p} \left[ e^{(t - t_0) A} x_0 - x_n \right] \\
		& = \frac{1}{\delta_n^p} \left[ \prod_{j = 0}^{n - 1} e^{\tau_j(n) A} - \prod_{j = 0}^{n - 1} \psi_j(n) \right] x_0 \\
		& = \frac{1}{\delta_n^p} \left[ \prod_{j = 0}^{n - 1} \big( \psi_j(n) + \omega_j(n) \big) - \prod_{j = 0}^{n - 1} \psi_j(n) \right] x_0.
	\end{align*}
	By expanding the first product, we get
	\begin{equation}
		\frac{\varepsilon_n}{\delta_n^p} = 
			\frac{1}{\delta_n^p} \left[ \sum_{j = 0}^{n - 1} \omega_j(n) \prod_{\substack{l = 0 \\ l \neq j}}^{n - 1} \psi_l(n) + \rho(n) \right] x_0, \label{epsilonn}
	\end{equation}
	where
	\begin{align*}
		\frac{\abs{\rho(n)}}{\delta_n^p} & \leq \frac{E}{\delta_n^p} \sum_{r = 2}^n \binom{n}{r} 
			\left( \sum_{k = p + 1}^\infty \frac{\big( C L(\gamma) \norm{A} \big)^k}{k! \cdot n^k} \right)^r \\
		& \leq \frac{E}{\delta_n^p} \sum_{r = 2}^n \underbrace{\frac{n (n - 1) \ldots (n - r + 1)}{r! \cdot n^r}}_{\leq 1}
			\left(n \sum_{k = p + 1}^\infty \frac{\big( C L(\gamma) \norm{A} \big)^k}{k! \cdot n^k} \right)^r \\
		& \leq \frac{E \cdot n^p}{\abs{t - t_0}^p} \sum_{r = 2}^n \left( \frac{E}{n^p} \right)^r 
			\leq \frac{E \cdot n^p}{\abs{t - t_0}^p} \left( \frac{1}{1 - \frac{E}{n^p}} - 1 - \frac{E}{n^p} \right) = 
				\frac{E}{\abs{t - t_0}^p} \cdot \frac{E^2}{n^p - E},
	\end{align*}
	which implies $\lim_{n \to \infty} \frac{\rho(n)}{\delta_n^p} x_0 = 0$.
	In the next steps we show, that we can omit the missing $\psi_j(n)$ (products) and the higher order terms of the $\omega_j(n)$'s in (\ref{epsilonn}) 
	as $n$ tends to infinity. By the use of
	\begin{align*}
		& \frac{1}{\delta_n^p} \Bigg\lVert \sum_{j = 0}^{n - 1} \omega_j(n) \cdot x_n - 
			\sum_{j = 0}^{n - 1} \omega_j(n) \prod_{\substack{l = 0 \\ l \neq j}}^{n - 1} \psi_l(n) \cdot x_0 \Bigg\rVert \\
		\leq & \sum_{j = 0}^{n - 1} \left[ \Bigg\lVert \frac{\omega_j(n)}{\delta_n^p} \Bigg\rVert \Bigg\lVert x_n - 
			\prod_{\substack{l = 0 \\ l \neq j}}^{n - 1} \psi_l(n) \cdot x_0 \Bigg\rVert \right] \leq
				 \sum_{j = 0}^{n - 1} \left[ \Bigg\lVert \frac{\omega_j(n)}{\delta_n^p} \Bigg\rVert \Bigg\lVert \prod_{l = 0}^{n - 1} \psi_l(n) - 
					\prod_{\substack{l = 0 \\ l \neq j}}^{n - 1} \psi_l(n) \Bigg\rVert \right] \norm{x_0} \\
		\leq & \sum_{j = 0}^{n - 1} \left[ \frac{\sum_{k = p + 1}^\infty \frac{(C L(\gamma) \norm{A})^k}{k! \cdot n^{k - p}}}{\abs{t - t_0}^p}
			\Bigg\lVert \sum_{k = 0}^p \frac{\big( \tau_j(n) A \big)^k}{k!} - I \Bigg\rVert 
				\Bigg\lVert \prod_{\substack{l = 0 \\ l \neq j}}^{n - 1} \psi_l(n) \Bigg\rVert \right] \norm{x_0} \\
		\leq & \sum_{j = 0}^{n - 1} \frac{E}{n \abs{t - t_0}^p} \cdot \frac{E}{n} \cdot E \norm{x_0} = \frac{E^3 \norm{x_0}}{\abs{t - t_0}^p} \cdot \frac{1}{n},
	\end{align*}
	it holds that
	\begin{align*}
		\lim_{n \to \infty} \frac{\varepsilon_n}{\delta_n^p} & = \lim_{n \to \infty} \frac{1}{\delta_n^p} \sum_{j = 0}^{n - 1} \omega_j(n) \cdot x_n \\
		& = \lim_{n \to \infty} \left[ \frac{1}{\delta_n^p} \sum_{j = 0}^{n - 1} \frac{\big( \tau_j(n) A \big)^{p + 1}}{(p + 1)!} + 
			\frac{1}{\delta_n^p} \sum_{j = 0}^{n - 1} \sum_{k = p + 2}^\infty \frac{\big( \tau_j(n) A \big)^k}{k!} \right] x_n.
	\end{align*}
	Taking into account that
	\begin{align*}
		\abs{ \frac{1}{\delta_n^p} \sum_{j = 0}^{n - 1} \sum_{k = p + 2}^\infty \frac{\big( \tau_j(n) A \big)^k}{k!}} & \leq
			\frac{1}{\abs{t - t_0}^p} \sum_{j = 0}^{n - 1} \sum_{k = p + 2}^\infty \frac{n^p \big(C L(\gamma) \norm{A} \big)^k}{k! \cdot n^k} \\
		& \leq \frac{1}{\abs{t - t_0}^p \cdot n^2} \sum_{j = 0}^{n - 1} E = \frac{E}{\abs{t - t_0}^p} \cdot \frac{1}{n}
	\end{align*}
	and $\lim_{n \to \infty} x_n = \Phi^{t, t_0} x_0 = e^{(t - t_0) A} x_0$, shows finally the desired claim.
\end{proof}
\paragraph{Note ($s$-stage (eRKM))} What happens if the number of stages of the (eRKM) is larger than its order of convergence?
	In the case of an $s$-stage (eRKM) of order $p < s \in \mathbb{N}$ (notice that $p > s$ is not possible in the explicit case) 
	one can derive a theorem analogous to Theorem \ref{thm:mainTheorem}. We have omitted this case for a better readability 
	of the preceding proof.
	The only difference to the results above is given by a constant $C \in \mathbb{C}^{s \times s}$ 
	(the $p + 1$ coefficient of the \textsc{Runge-Kutta}-polynomial $P(\tau A)$), such that
	\begin{equation*}
		\lim_{n \to \infty} \frac{\varepsilon_n}{\delta_n^p} = 
			\lim_{n \to \infty} \frac{1}{\delta_n^p} \sum_{j = 0}^{n - 1} \big( \tau_j(n) \big)^{p + 1} \cdot 
				 e^{(t - t_0) A} \left[\frac{A^{p + 1}}{(p + 1)!} - C \right] x_0.
	\end{equation*}
	The corresponding proof is achieved by some additional estimations of higher order terms of $\frac{\varepsilon_n}{\delta_n^p}$, 
	analogously to the ones seen above.

\paragraph{Note (implicit (RKM)s)}
	Analogously to the proof above, one can show, that for an $s$-stage implicit (RKM) method, the structure of the error expansion 
	(\ref{eqn:errorExpansion}) is also of the form
	\begin{equation*}
		\lim_{n \to \infty} \frac{1}{\delta_n^p} \sum_{j = 0}^{n - 1} \big( \tau_j(n) \big)^{p + 1} \cdot C(t, t_0, A) x_0,
	\end{equation*}
	where $C(t, t_0, A) \in \mathbb{C}^{s \times s}$.

\begin{proposition} \label{rem:superconvergence}
	Theorem \ref{thm:mainTheorem} yields
	\begin{equation*}
		\sum_{j = 0}^{n - 1} \big( \tau_j(n) \big)^{p + 1} = 0, ~ \forall n \geq N \in \mathbb{N}_0 \quad \Rightarrow \quad
			\lim_{n \to \infty} \frac{\varepsilon_n}{\delta_n^p} = \lim_{n \to \infty} \frac{\Phi^{t, t_0} x_0 - x_n}{\delta_n^p} = 0.
	\end{equation*}	
	In the context of the premises of Theorem \ref{thm:mainTheorem}, this statement provides us a condition, 
	which increases the order of a (RKM) of order $p \in \mathbb{N}$ applied to (\ref{lCIVP}) from $p$ to $p + 1$ at the terminal point.
\end{proposition}
\subsection{Superconvergent paths} \label{subsec:superconvergentPath}
With Remark \ref{rem:superconvergence} in mind, our next goal is obvious: we want to construct complex detours that provide us 
with superconvergence for linear ODEs. We consider an $s$-stage (RKM) of convergence order $p \in \mathbb{N}$ and use it as
a method to compute an approximation of the complex flow $\Phi^{t, t_0} x_0$ corresponding to problem (\ref{lCIVP}), where $t \in \mathbb{C}$.
To achieve superconvergence we are interested in time grid families $\left(\Delta = \left[t_0^\Delta, \ldots, t_{n_\Delta}^\Delta\right]\right)_{n \in \mathbb{N}}$, with
\begin{enumerate} \label{criterions}
	\item	$t_0 = t_0^\Delta$ and $t = t_{n_\Delta}^\Delta$  (in other words $\sum_{j = 0}^{n_\Delta - 1} \tau_j^\Delta = t - t_0$),
	\item	$\sum_{j = 0}^{n_\Delta - 1} \left(\tau_j^\Delta\right)^{p + 1} = 0$ (this provides superconvergence) and
	\item	$\tau_\Delta \leq \frac{D}{n_\Delta}$, where $D \in \mathbb{R}^+$ (prerequisite of Theorem \ref{thm:mainTheorem})
\end{enumerate}
Fortunately the second criterion can be easily interpreted geometrically. This helps us to get a canonical family of time grids
for every $p \in \mathbb{N}$:

We use the fact that all $n$-th roots of unity sum up to zero. Setting  $\zeta_n(j) := e^{2 \pi i \frac{j}{n}}, ~ j \in \mathbb{Z}$, it holds for every $\alpha \in \mathbb{C}$, that
\begin{equation*}
	\sum_{j = 0}^{n - 1} \alpha \zeta_n(j) = 0.
\end{equation*}

Now we choose the time grid $\Delta$ in a way, that $(\tau_j^\Delta)^{p+1}$, $j \in \{0, \ldots, n_\Delta - 1\}$, becomes such a 
collection of roots of unity. As illustrated by Figure \ref{fig:nthRootsOfUnity}, $n$ consecutive $n (p + 1)$-th roots of 
unity $\zeta_{n (p + 1)}(k), \ldots, \zeta_{n (p + 1)}(k + (n - 1))$, are transformed by complex exponentiation $z \mapsto z^{p + 1}$ to the
set of $n$-th roots of unity.
\begin{figure}[hbt]
	\begin{center}
		\psset{xunit=2.8cm,yunit=2.8cm,runit=2.8cm}
		\begin{pspicture}(0,-0.1)(5,2.1)
			\put(2.41,1.27){$z^{p + 1}$}
			\put(2.4,1.1){\Huge{$\curvearrowright$}}
			\put(1.75,1.75){$\mathbb{C}$}
			\psline[](0,1)(2,1)
			\psline[](1,0)(1,2)
			\psline[](1.8,0.97)(1.8,1.03)
			\psline[](0.2,0.97)(0.2,1.03)
			\psline[](0.97,1.8)(1.03,1.8)
			\psline[](0.97,0.2)(1.03,0.2)
			\put(4.75,1.75){$\mathbb{C}$}
			\psline[](3,1)(5,1)
			\psline[](4,0)(4,2)
			\psline[](4.8,0.97)(4.8,1.03)
			\psline[](3.2,0.97)(3.2,1.03)
			\psline[](3.97,1.8)(4.03,1.8)
			\psline[](3.97,0.2)(4.03,0.2)
			\dotnode(1.247213595,1.760845213){0}
			\nput{80}{0}{${}_{\zeta_{15}(3)}$}
			\psline[linewidth=1pt,arrowsize=1pt 2,arrowlength=3,arrowinset=0.1]{->}(1,1)(1.247213595,1.760845213)
			\dotnode(0.9163772293,1.795617516){1}
			\nput{110}{1}{${}_{\zeta_{15}(4)}$}
			\psline[linewidth=1pt,arrowsize=1pt 2,arrowlength=3,arrowinset=0.1]{->}(1,1)(0.9163772293,1.795617516)
			\dotnode(0.6000000000,1.692820323){2}
			\nput{120}{2}{${}_{\zeta_{15}(5)}$}
			\psline[linewidth=1pt,arrowsize=1pt 2,arrowlength=3,arrowinset=0.1]{->}(1,1)(0.6000000000,1.692820323)
			\dotnode(0.3527864045,1.470228202){3}
			\nput{140}{3}{${}_{\zeta_{15}(6)}$}
			\psline[linewidth=1pt,arrowsize=1pt 2,arrowlength=3,arrowinset=0.1]{->}(1,1)(0.3527864045,1.470228202)
			\dotnode(0.2174819193,1.166329352){4}
			\nput{170}{4}{${}_{\zeta_{15}(7)}$}
			\psline[linewidth=1pt,arrowsize=1pt 2,arrowlength=3,arrowinset=0.1]{->}(1,1)(0.2174819193,1.166329352)
			\dotnode(3.352786405,0.5297717970){5}
			\nput{225}{5}{${}_{\zeta_{5}(3)}$}
			\psline[linewidth=1pt,arrowsize=1pt 2,arrowlength=3,arrowinset=0.1]{->}(4,1)(3.352786405,0.5297717970)
			\dotnode(4.247213596,0.2391547870){6}
			\nput{-70}{6}{${}_{\zeta_{5}(4)}$}
			\psline[linewidth=1pt,arrowsize=1pt 2,arrowlength=3,arrowinset=0.1]{->}(4,1)(4.247213596,0.2391547870)
			\dotnode(4.800000001,0.9999999997){7}
			\nput{45}{7}{${}_{\zeta_{5}(0)}$}
			\psline[linewidth=1pt,arrowsize=1pt 2,arrowlength=3,arrowinset=0.1]{->}(4,1)(4.800000001,0.9999999997)
			\dotnode(4.247213595,1.760845213){8}
			\nput{60}{8}{${}_{\zeta_{5}(1)}$}
			\psline[linewidth=1pt,arrowsize=1pt 2,arrowlength=3,arrowinset=0.1]{->}(4,1)(4.247213595,1.760845213)
			\dotnode(3.352786404,1.470228201){9}
			\nput{150}{9}{${}_{\zeta_{5}(2)}$}
			\psline[linewidth=1pt,arrowsize=1pt 2,arrowlength=3,arrowinset=0.1]{->}(4,1)(3.352786404,1.470228201)
		\end{pspicture}
		\caption{Transformation for $n = 5$, $p = 2$ and $k = 3$.}
		 \label{fig:nthRootsOfUnity}
	\end{center}
\end{figure}
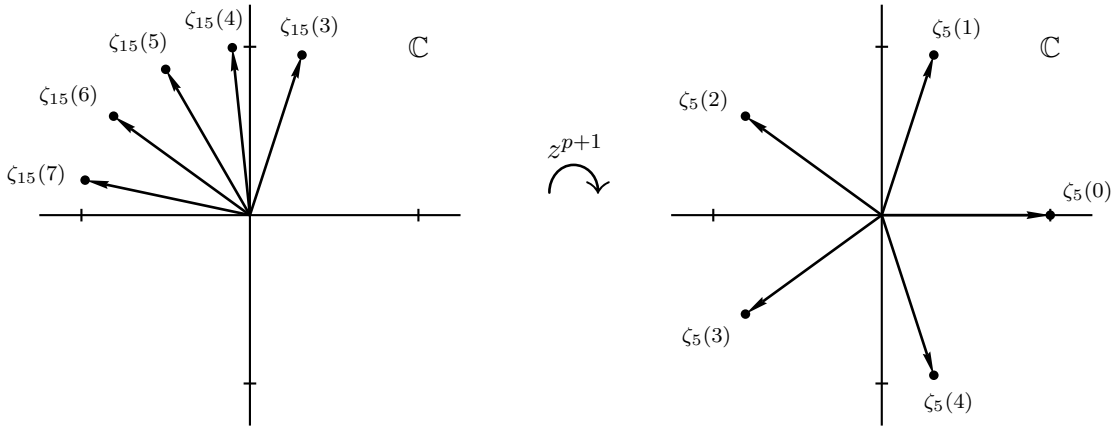
\begin{lemma}
	Let $n, p \in \mathbb{N}, ~ k \in \mathbb{Z}$ and $\alpha \in \mathbb{C}$. Then it holds that
	\begin{equation*}
		\sum_{j = k}^{k + (n - 1)} \big(\alpha \zeta_{n (p + 1)}(j)\big)^{p + 1} = 0
	\end{equation*}
\end{lemma}
\begin{proof}
	It holds that
	\begin{align*}
		\sum_{j = k}^{k + (n - 1)} \big(\alpha \zeta_{n (p + 1)}(j)\big)^{p + 1} & = 
			\alpha^{p + 1} \sum_{j = k}^{k + (n - 1)} \left(e^{2 \pi i \frac{j}{n (p + 1)}}\right)^{p + 1} = 
				\alpha^{p + 1}\sum_{j = k}^{k + (n - 1)} e^{2 \pi i \frac{j}{n}} \\
		& = \alpha^{p + 1} \sum_{j = k}^{n - 1} e^{2 \pi i \frac{j}{n}} + \alpha^{p + 1} \sum_{j = n}^{k + (n - 1)} e^{2 \pi i \frac{j}{n}} \\
		&= \alpha^{p + 1} \sum_{j = k}^{n - 1} e^{2 \pi i \frac{j}{n}} + \alpha^{p + 1} \sum_{l = 0}^{k - 1} e^{2 \pi i \frac{l + n}{n}} \\
		& = \alpha^{p + 1} \sum_{j = 0}^{n - 1} e^{2 \pi i \frac{j}{n}} = 0.
	\end{align*}
\end{proof}

In conclusion, one way to build a superconvergent time grid $\Delta = \left[t_0^\Delta, \ldots, t_{n_\Delta}^\Delta\right]$ is to take
time steps $\tau_j^\Delta$ represented by $n$ consecutive $n (p + 1)$-th roots of unity, which have been suitable globally scaled and 
rotated by a factor $\alpha \in \mathbb{C}$, such that $\sum_{j = 0}^{n_\Delta - 1} \tau_j^\Delta = t - t_0$ 
(compare Figure \ref{fig:sumNthRootsOfUnity}). A proper choice for this factor is
\begin{equation*}
	\alpha := \frac{t - t_0}{\sum_{j = k}^{k + (n - 1)} \zeta_{n (p + 1)}(j)}.
\end{equation*}
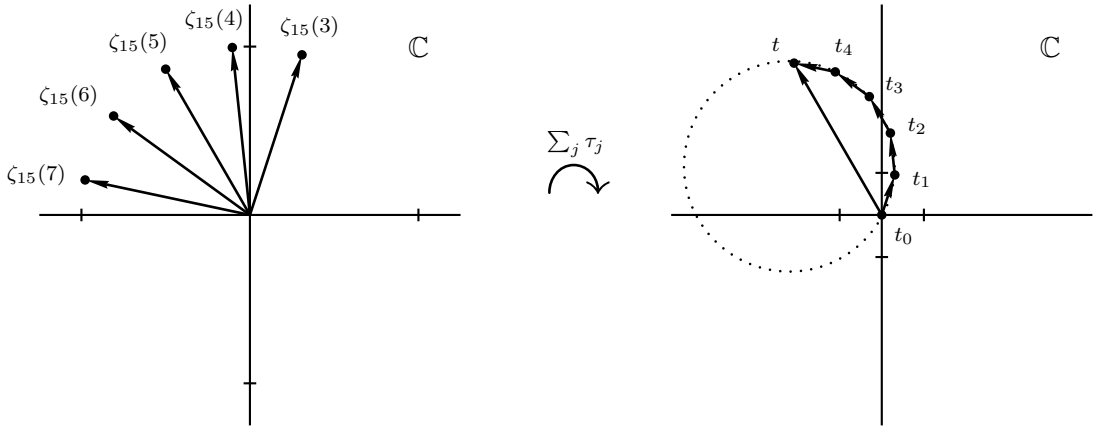
\begin{figure}[hbt]
	\begin{center}
		\psset{xunit=2.8cm,yunit=2.8cm,runit=2.8cm}
		\begin{pspicture}(0,-0.1)(5,2.1)
			\put(2.41,1.35){${}_{\sum_{j} \tau_j}$}
			\put(2.4,1.1){\Huge{$\curvearrowright$}}
			\put(1.75,1.75){$\mathbb{C}$}
			\psline[](0,1)(2,1)
			\psline[](1,0)(1,2)
			\psline[](1.8,0.97)(1.8,1.03)
			\psline[](0.2,0.97)(0.2,1.03)
			\psline[](0.97,1.8)(1.03,1.8)
			\psline[](0.97,0.2)(1.03,0.2)
			\put(4.75,1.75){$\mathbb{C}$}
			\psline[](3,1)(5,1)
			\psline[](4,0)(4,2)
			\psline[](4.2,0.97)(4.2,1.03)
			\psline[](3.8,0.97)(3.8,1.03)
			\psline[](3.97,1.2)(4.03,1.2)
			\psline[](3.97,0.8)(4.03,0.8)
			\dotnode(1.247213595,1.760845213){0}
			\nput{80}{0}{${}_{\zeta_{15}(3)}$}
			\psline[linewidth=1pt,arrowsize=1pt 2,arrowlength=3,arrowinset=0.1]{->}(1,1)(1.247213595,1.760845213)
			\dotnode(0.9163772293,1.795617516){1}
			\nput{110}{1}{${}_{\zeta_{15}(4)}$}
			\psline[linewidth=1pt,arrowsize=1pt 2,arrowlength=3,arrowinset=0.1]{->}(1,1)(0.9163772293,1.795617516)
			\dotnode(0.6000000000,1.692820323){2}
			\nput{120}{2}{${}_{\zeta_{15}(5)}$}
			\psline[linewidth=1pt,arrowsize=1pt 2,arrowlength=3,arrowinset=0.1]{->}(1,1)(0.6000000000,1.692820323)
			\dotnode(0.3527864045,1.470228202){3}
			\nput{140}{3}{${}_{\zeta_{15}(6)}$}
			\psline[linewidth=1pt,arrowsize=1pt 2,arrowlength=3,arrowinset=0.1]{->}(1,1)(0.3527864045,1.470228202)
			\dotnode(0.2174819193,1.166329352){4}
			\nput{170}{4}{${}_{\zeta_{15}(7)}$}
			\psline[linewidth=1pt,arrowsize=1pt 2,arrowlength=3,arrowinset=0.1]{->}(1,1)(0.2174819193,1.166329352)
			\pscircle[linecolor=black,linestyle=dotted,linewidth=1pt](3.561,1.231){0.5}
			\dotnode(4,1){t0}
			\nput{-45}{t0}{${}_{t_0}$}
			\dotnode(4.061803399,1.190211303){5}
			\nput{-10}{5}{${}_{t_1}$}
			\psline[linewidth=1pt,arrowsize=1pt 2,arrowlength=3,arrowinset=0.1]{->}(4,1)(4.061803399,1.190211303)
			\dotnode(4.040897706,1.389115682){6}
			\nput{15}{6}{${}_{t_2}$}
			\psline[linewidth=1pt,arrowsize=1pt 2,arrowlength=3,arrowinset=0.1]{->}(4.061803399,1.190211303)(4.040897706,1.389115682)
			\dotnode(3.940897706,1.562320763){7}
			\nput{25}{7}{${}_{t_3}$}
			\psline[linewidth=1pt,arrowsize=1pt 2,arrowlength=3,arrowinset=0.1]{->}(4.040897706,1.389115682)(3.940897706,1.562320763)
			\dotnode(3.779094307,1.679877813){8}
			\nput{70}{8}{${}_{t_4}$}
			\psline[linewidth=1pt,arrowsize=1pt 2,arrowlength=3,arrowinset=0.1]{->}(3.940897706,1.562320763)(3.779094307,1.679877813)
			\dotnode(3.583464787,1.721460151){9}
			\nput{135}{9}{${}_{t}$}
			\psline[linewidth=1pt,arrowsize=1pt 2,arrowlength=3,arrowinset=0.1]{->}(3.779094307,1.679877813)(3.583464787,1.721460151)
			\psline[linecolor=black,linewidth=1pt,arrowsize=1pt 2,arrowlength=3,arrowinset=0.1]{->}(4,1)(3.583464787,1.721460151)
		\end{pspicture}
		\caption{Counterparts for $n = 5$, $p = 2$ and $k = 3$ (no global scaling and rotation).}
		 \label{fig:sumNthRootsOfUnity}
	\end{center}
\end{figure}

To achieve a superconvergent family of time grids, it is easier to focus on the time steps $\tau_j^\Delta$ as done above. 
For coding and proof, we now change from the time steps to the time grid elements $t_j^\Delta, ~ j \in \{0, \ldots, n_\Delta\}$. 

As one can see by example in Figure \ref{fig:sumNthRootsOfUnity}, a \emph{superconvergent} choice of the time steps $\tau_j^\Delta$, induces that
the time grid elements $t_j^\Delta$ are located on a circle segment, depending on the order $p \in \mathbb{N}$ of the chosen method,
connecting $t_0, t \in \mathbb{C}$.
Figure \ref{fig:sumNthRootsOfUnity} shows the situation for $p = 2$. There, the corresponding time grid elements lie on a third segment of a circle.
Compare this to the half circle segment of the introductory example, where the explicit \textsc{Euler} method has been used ($p = 1$).

In conclusion, if the order of the method is $p \in \mathbb{N}$, locating the time grid elements $t_j^\Delta$ equidistantly on a $(p + 1)$-th
circle segment, connecting $t_0, t \in \mathbb{C}$, provides the desired superconvergence effect. The following theorem formalizes this idea.
\begin{theorem} \label{thm:superconvergentPath}
	Let $\gamma^* := \gamma_{t_0, t}^p$, where
	\begin{equation} \label{eq:gammaFamily}
		\gamma^p_{t_0, t}: [0,1] \to \mathbb{C}, ~ x \mapsto \frac{t_0 - t}{2 i \sin \left(\frac{\pi}{p + 1} \right)} \left[ e^{i \pi \left( \frac{1 - 2 x}{p + 1} \right)} - 
			\cos \left( \frac{\pi}{p + 1} \right) \right] + \frac{t_0 + t}{2}.
	\end{equation}
	Furthermore, let $n \in \mathbb{N}$ and $x_{\Delta^{\gamma^*}_n}$ be the grid function generated by a $s$-stage (RKM) of order 
	$p \in \mathbb{N}$ applied to (\ref{lCIVP}). Then
	\begin{equation*}
		t_0 = \gamma^*(0) \quad \text{and} \quad t = \gamma^*(1).
	\end{equation*}
	Furthermore,
	\begin{equation*}
		\lim_{n \to \infty} \frac{\Phi^{t, t_0} x_0 - x_{\Delta^{\gamma^*}_n}(t)}{\delta_n^p} = 0.
	\end{equation*}
\end{theorem}
\begin{proof}
	By the use of \textsc{Euler}'s formula, a simple calculation shows that $t_0 = \gamma^*(0)$ and $t = \gamma^*(1)$.
	Moreover
	\begin{equation*}
		\tau_j^{\Delta^{\gamma^*}_n} = 
			\frac{t_0 - t}{2 i \sin{\frac{\pi}{p + 1}}} \left( e^{i \pi \frac{1 - 2 \frac{j + 1}{n}}{p + 1}} - e^{i \pi \frac{1 - 2 \frac{j}{n}}{p + 1}} \right),
				\quad \forall j \in \{0, \ldots, n - 1\},
	\end{equation*}
	so it is not hard to see, that
	\begin{equation*}
		\abs{\tau_j^{\Delta^{\gamma^*}_n}} = \text{const}_{t_0, t, p} \cdot \left| e^{-\frac{2 \pi i}{n (p + 1)}} - 1 \right|, \quad \forall j \in \{0, \ldots, n - 1\},
	\end{equation*}
	which implies\footnote{Compare to criterion 3 on page \pageref{criterions}.} $\delta_n \in \Theta(n^{-1})$ and additionally the equidistance of the time steps 
	$\tau_j^{\Delta^{\gamma^*}_n}$. 
	Finally, it is sufficient to show, that
	\begin{equation*}
		\sum_{j = 0}^{n - 1} \left(\tau_j^{\Delta^{\gamma^*}_n}\right)^{p + 1} = 
			\sum_{j = 0}^{n - 1} \left(t_{j + 1}^{\Delta^{\gamma^*}_n} - t_j^{\Delta^{\gamma^*}_n}\right)^{p + 1} = 0.
	\end{equation*}
	It holds that
	\begin{align*}
		\sum_{j = 0}^{n - 1} \left(\tau_j^{\Delta^{\gamma^*}_n}\right)^{p + 1} = 0 \quad \Leftarrow \quad & 
			\sum_{j = 0}^{n - 1} \left( e^{-\frac{2 \pi i}{p + 1} \frac{j + 1}{n}} - e^{-\frac{2 \pi i}{p + 1} \frac{j}{n}}\right)^{p + 1} = 0 \\
				\quad \Leftrightarrow \quad & \sum_{j = 0}^{n - 1} \left( e^{-\frac{2 \pi i}{p + 1} \frac{j }{n}  -\frac{2 \pi i}{p + 1} 
					\frac{1 }{n}} - e^{-\frac{2 \pi i}{p + 1} \frac{j}{n}}\right)^{p + 1} = 0 \\
		\quad \Leftrightarrow \quad & \left( e^{-\frac{2 \pi i}{p + 1} \frac{1}{n}} - 1 \right)^{p + 1} 
			\underbrace{\sum_{j = 0}^{n - 1} e^{-2 \pi i \frac{j }{n}}}_{=:(*)} = 0.
	\end{align*}
	Since $(*)$ is the sum of the $n$-th roots of unity, the desired claim follows by using the expansion given by Theorem \ref{thm:mainTheorem}. 
\end{proof}

The theorem above yields immediately the following 
\begin{corollary}[superconvergence of (RKM)s] \label{cor:superconvergence}
	Given a \textsc{Runge-Kutta} method of order $p \in \mathbb{N}$ applied to (\ref{lCIVP}) along $\gamma^*$.
	The approximation at the terminal point is of order $p + 1$.
\end{corollary}
\paragraph{Note} 
  As complex conjugation do not affect the argumentation so far, the previous corollary holds also
	for $\gamma^* := \overline{\gamma^*}$, where $\overline{\gamma^*}(t) := \overline{\gamma^*(t)}$ for all $t \in [0,1]$.
\section{Nonlinear case}
As we have seen in the last section, \textsc{Runge-Kutta} methods are superconvergent along certain paths of integration,
if the right-hand side of the corresponding (CIVP) is  linear. Having this fact in mind, it is natural to ask whether this feature transfers to the case
of an arbitrary right-hand side $f$. Again our study was triggered by several numerical experiments. These experiments suggested an
adapted version of the above superconvergence statement, that also holds locally in the nonlinear case.
The main idea is to study the order of convergence not for $n_\Delta \to \infty$, as in the linear case, but for $t \to t_0$. Let now 
$n_\Delta = k$ be a fixed number of time steps, $k \in \mathbb{N}$, and let $h := t - t_0$.

We will see, that according to the integration along a time grid $\Delta$ induced by $\gamma_{t_0,t}^p$, it holds that
\begin{equation*}
	\varepsilon_n \in \mathcal{O}\left(\abs{h}^{p + 1}\right),
\end{equation*}
in contrast to the linear case, where $\varepsilon_n \in \mathcal{O}\left(\frac{1}{n_\Delta^{p + 1}}\right)$. Roughly speaking, 
the superconvergence effect is independent of $\abs{t - t_0}$ in the linear case.

Figure \ref{fig:microPathSegment} represents the situation in which a numerical integration along the discretization $\Delta$ is done by 
two $\textsc{Runge-Kutta}$ iterations ($k = 2$). 
If we regard this path as a scaled and rotated copy of a normalized\footnote{This means a $(p + 1)$-th circle segment 
$\gamma_{0,1}^p$ connecting $0$ and $1$, $p \in \mathbb{N}$.} circle segment, we can express the two time steps of $\Delta$ as follows. 
\begin{equation*}
	t_1^\Delta - t_0^\Delta = \sigma_1 h \quad \text{and} \quad t_2^\Delta - t_1^\Delta = \sigma_2 h,
\end{equation*} 
where $\sigma_1$ and $\sigma_2 \in \mathbb{C}$ are the two time steps of the grid $\Delta_2^{\gamma_{0,1}^1}$.\\
\begin{figure}[h]
	\begin{center}
	\psset{xunit=0.9cm,yunit=0.9cm,runit=0.9cm}
	\begin{pspicture}(0,0)(6,4.5)
		\psarc[linestyle=dashed,linewidth=1pt](3,0.5){3}{0}{180}
		\dotnode[dotstyle=diamond*,dotscale=1.5 1.5](0,0.5){tj}
		\nput{-90}{tj}{$t_0^{\Delta}$}
		\dotnode[dotstyle=diamond*,dotscale=1.5 1.5](6,0.5){tjplus}
		\nput{-90}{tjplus}{$t_2^{\Delta}$}
		\dotnode[dotstyle=diamond*,dotscale=2 1.5](3,3.5){tjsigma}
		\nput{90}{tjsigma}{$t_1^{\Delta}$}
		\psline[linewidth=1pt,arrowsize=1pt 2,arrowlength=3,arrowinset=0.1]{->}(0,0.5)(6,0.5)
		\psline[linewidth=1pt,arrowsize=1pt 2,arrowlength=3,arrowinset=0.1]{->}(0,0.5)(3,3.5)
		\psline[linewidth=1pt,arrowsize=1pt 2,arrowlength=3,arrowinset=0.1]{->}(3,3.5)(6,0.5)
		\put(0.88,2.28){$\sigma_1 h$}
		\put(4.35,2.28){$\sigma_2 h$}
		\put(2.85,0.7){$h$}
	\end{pspicture}
	\caption{	Scaled path segment from $t_0^{\Delta}$ to $t_2^{\Delta}$ for $p = 1, ~ k = 2$.}
	\label{fig:microPathSegment}
	\end{center}
\end{figure}
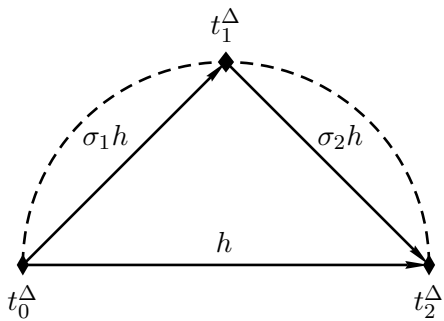

The general case, $k > 2$, behaves analogously. For $l \in \{1, \ldots, k\}$, let $\sigma_l$ be the $l$-th time step of the
normalized grid $\Delta_k^{\gamma_{0,1}^p}$. Due to this fact, the corresponding scaled time steps can be expressed by
$\sigma_l h$, where $h := t - t_0$.

\subsection{Composition methods}
Now, we draw a connection between compositions methods and our idea of numerical integration along complex paths.
Let $\Psi^\tau := \Psi^{0 + \tau, 0}$ be the underlying autonomous\footnote{This is w.l.o.g. no restriction.} discrete evolution of the basic (RKM). 
If we introduce the $k$-term composition method
\begin{equation} \label{epn:Upsilon}
	\Upsilon^h := \Psi^{\sigma_k h} \circ \ldots \circ \Psi^{\sigma_1 h}, 
\end{equation}
the process of constructing $x_k := x_{\Delta}\left(t_k^{\Delta}\right)$ from 
$x_0 := x_{\Delta}\left(t_0^{\Delta}\right)$, where $\Delta := \Delta_k^{\gamma_{0, h}^p}$, is given by
\begin{equation*}
	x_k = \Upsilon^{h} x_0.
\end{equation*}
By the  superconvergence conditions of Section \ref{subsec:superconvergentPath}, our $\sigma_1, \ldots, \sigma_k$ satisfy the two equations
\begin{equation} \label{eqn:orderConditions}
	\sigma_1 + \ldots + \sigma_k = 1 \quad \text{and} \quad \sigma_1^{p + 1} + \ldots + \sigma_k^{p + 1} = 0.
\end{equation}
These conditions are exactly the criteria for a classical theorem on composition methods. 
From \cite[p. 39, Theorem 4.1]{HLW02} it follows, that $\Upsilon^h$ is a one-step method of convergence order at least $p + 1$.

\paragraph{Remark}
	The order conditions (\ref{eqn:orderConditions}) can be found in \cite{HLW02}. 
	There, some specific composition methods have been introduced by explicitly solving the above order conditions over the reals.
	\cite{HO08} and \cite{V08} also stated methods, by solving these equations over the complex numbers. 
	In contrast to them, our starting point has not been (\ref{eqn:orderConditions}). The adaption
	of superconvergent paths $\gamma_{t_0, t}^p$ from the linear theorem, solves implicitly the sufficient order conditions.

\subsection{Iterations}
The method given by (\ref{epn:Upsilon}) is called a $k$-term composition method. If the underlying discrete evolution $\Psi$ has got
order $p \in \mathbb{N}$, it follows that $\Upsilon$ is a one-step method of order at least $p + 1$. 
As concatenation of (RKM)s is a (RKM) again, there is no problem to regard $\Upsilon$ as a new basic (RKM). $\Upsilon$ can again be used in the same way
to construction a composition method -- this time of order at least $p+2$. 
Applying this process iteratively, it is possible to generate methods of arbitrary order of convergence. 

We start with a (RKM) $\Psi$ of order $p$. Iteration, leads to methods $\Upsilon^h_r$ (parameter $k$ omitted in this notation), 
which are recursively defined by
\begin{align*}
	\Upsilon^h_0 & := \Psi^h, \\
	\Upsilon^h_{r + 1} & := \Upsilon_r^{\Delta_r},
\end{align*}
where $\Delta_r := \Delta^{\gamma^{p(r)}_{0, h}}_k$ and $r \geq 0$. $\Upsilon^h_r$ is a method of order $p(r)$. Here
\begin{align*}
	p(0) & := p, \\
	p(r + 1) & := p(r) + g,
\end{align*}
where $g \in \mathbb{N}$ - the gain of order of convergence - depends on $\Psi$. 
For example, if $\Psi$ is the discrete evolution corresponding to a (eRKM), $g = 1$. If $\Psi$ is the discrete evolution of a symmetric method, $g = 2$.

With our geometric picture in mind, we now can reduce the method $\Upsilon_r$ to the basic method $\Psi$, applied along a certain path. 
The previous recursion for $\Upsilon_r$ implies, that 
$\Upsilon^h_r x_0 = \Psi^\Gamma x_0$, where $\Gamma$ is a recursively defined time grid depending on
the parameters $r$ (depth of recursion), $k$ (number of time steps), $p$ (order of $\Psi$) and $h$ (global step size). 
We exemplify $\Gamma$ for several choices of $r$ in Figure \ref{fig:gammaFractal}. As one can see, $\Gamma$ becomes more and more a
fractal-like structure. Furthermore the time steps of $\Gamma$ comply with the coefficients $\sigma_{r,j}$, $j \in \{1, 2\}$ 
in \cite[p. 5, ``two term composition'']{HO08} ($k > 2$ analogously). 
\begin{sidewaysfigure}[!h]
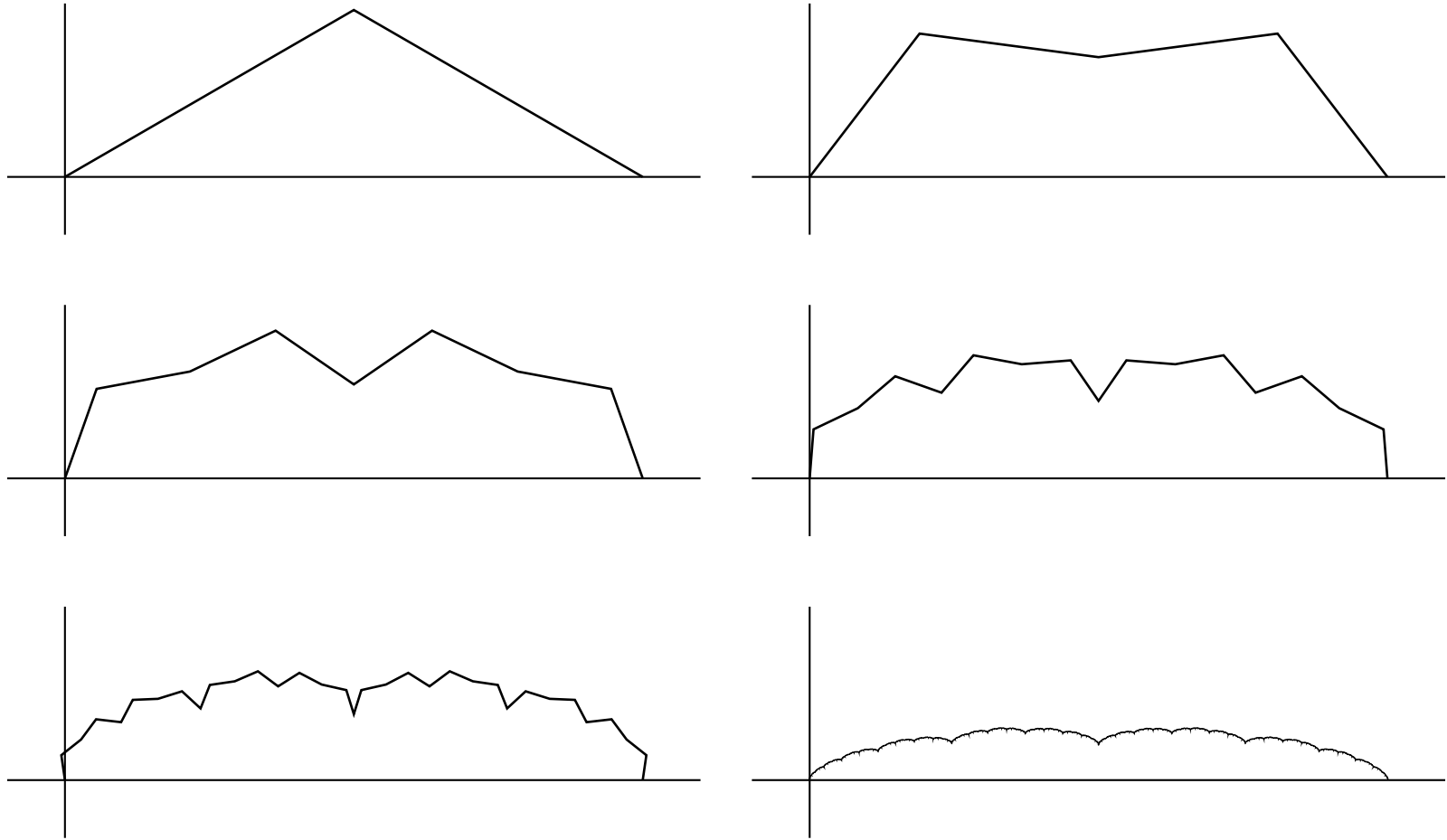

	\begin{center}
		\psset{xunit=8.5cm,yunit=8.5cm,runit=8.5cm}

	\end{center}
	\caption{	Time grid $\Gamma$ for $r = 1, 2, 3, 4, 5$ and $11$ ($g = h = 1$ and $p = k = 2$).}
	\label{fig:gammaFractal}
\end{sidewaysfigure}
\paragraph{Remark}
As mentioned above our approach leads to the same class of methods, already discovered in \cite{HO08} and \cite{V08}.
In contrast to the construction of $k$-term composition methods by solving the system of order conditions (\ref{eqn:orderConditions}),
we explicitly get these methods by using the basic (RKM) along special complex time grids. The iteration done in \cite{HO08} and \cite{V08},
to construct higher order methods based on a certain (RKM), transfers (in our geometric point of view) 
to an application of this (RKM) along the recursively defined time grid $\Gamma$.
\subsection{A complex orbit}
As a benchmark and illustration of the ideas we just introduced, we consider a classical example of celestial mechanics - the
restricted three body problem. The corresponding system of differential equations 
\begin{align} \label{eq:arenstorfSystem}
	\ddot{x}_1 & = x_1 + 2 \dot{x}_2 - \hat{\mu} \frac{x_1 + \mu}{\sqrt{\big((x_1 + \mu)^2 + x_2^2\big)^3}}
		- \mu \frac{x_1 - \hat{\mu}}{\sqrt{\big((x_1 - \hat{\mu})^2 + x_2^2\big)^3}}, \notag \\ \\
	\ddot{x}_2 & = x_2 - 2 \dot{x}_1 - \hat{\mu} \frac{x_2}{\sqrt{\big((x_1 + \mu)^2 + x_2^2\big)^3}}
		- \mu \frac{x_2}{\sqrt{\big((x_1 - \hat{\mu})^2 + x_2^2\big)^3}}, \notag
\end{align}
is motivated by the motion of a satellite with respect to the gravitation potential induced by the moon and the earth.
Thereby $\mu := 0.012277471$ is the ratio of the moon corresponding to the mass of the total system and $\hat{\mu} := 1 - \mu$.
Furthermore, the mass of the satellite could be neglected and $\big(x_1(t), x_2(t)\big)^T$ represents the vector of the satellite's coordinates 
(the motion stays in a plane) in terms of a ``rotating'' coordinate system, whose origin represents the gravitational center of the moon and the earth 
and in which both celestial bodies stay on predefined points on the $x_1$-axis. 

Due to the american mathematician \textsc{R. Arenstorf}\footnote{Compare \cite{A63}. At this point we want to mention, that \textsc{Arenstorf}
showed the existence of such orbits by analytic continuation of parameterized two-body problems, for which the exact solutions (conic sections) 
are explicitly known.}, we know, that there exists a periodic solution
to the corresponding initial value problem given by the system (\ref{eq:arenstorfSystem}) and a suitable chosen initial value.
As illustrated by Figure \ref{fig:arenstorfOrbit}, the error at the terminal point (integration from $t_0  = 0$ to $T := 17.065216560157960$, where 
$T$ is the period of the corresponding exact orbit) by the use of the explicit \textsc{Euler} method (discrete evolution denoted by $\Psi$)
is more than $55$ times larger than with the composition method $\Upsilon_1$ ($k = 2$).
Here, both approaches have been calculated by using equidistant time grids with the same number of time steps.
\begin{sidewaysfigure}[h]
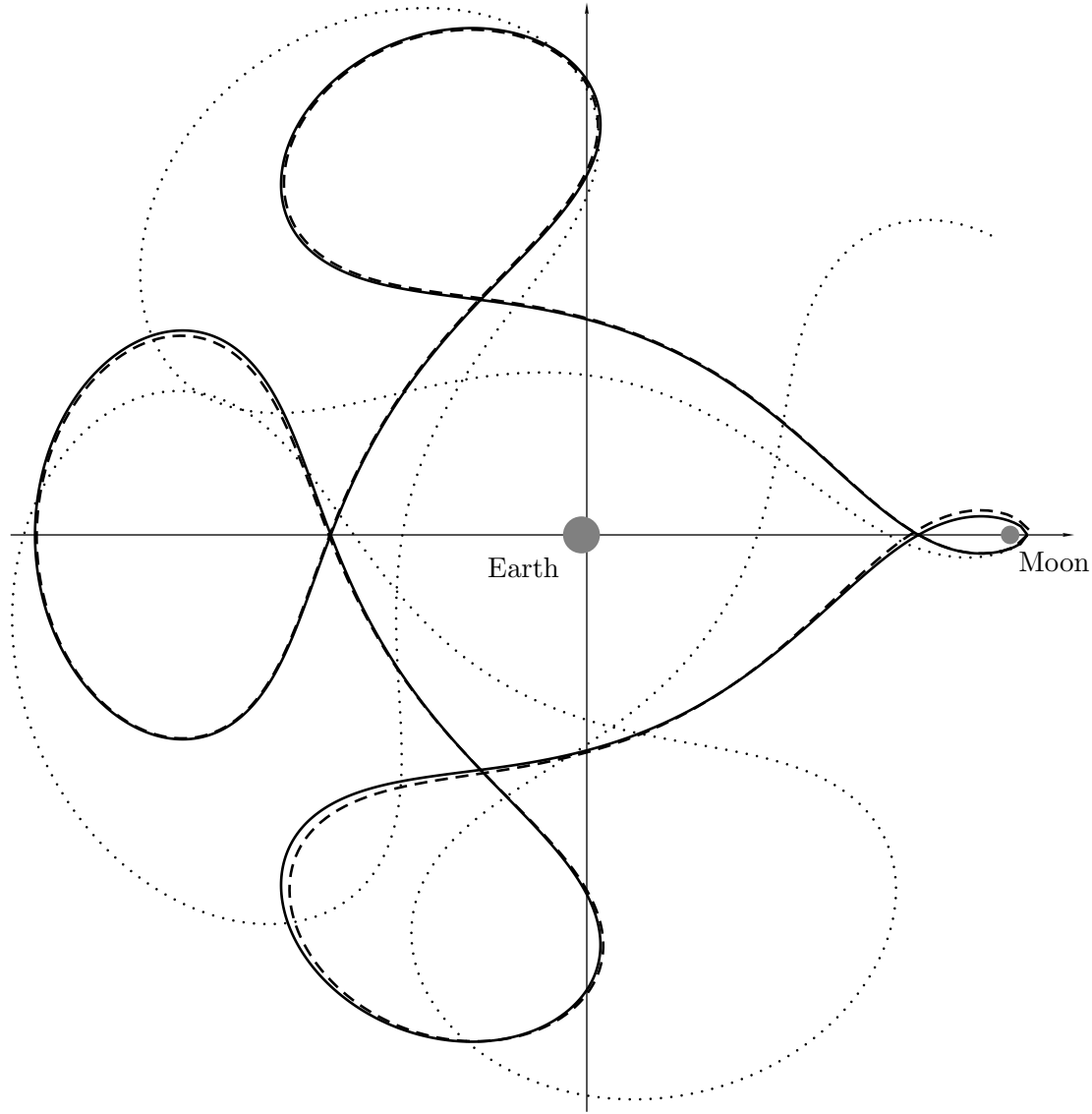

	\begin{center}
			\psset{xunit=6cm,yunit=6cm,runit=6cm}

	\end{center}
	\caption{	The ``exact'' \textsc{Arenstorf} orbit (period $T = 17.065216560157960$, corresponding initial value 
            $\big(x_1(0), x_2(0), \dot{x}_1(0), \dot{x}_2(0)\big)^T = \big(0.9940, 0.0, 0.0, -2.001585106379080\big)^T$)
						given by $100000$ real-valued equidistant steps of \textsc{Dopri5} (5-th order (eRKM), solid), 
						approximation by $100000$ equidistant real-valued steps of the explicit \textsc{Euler} method $\Psi$ (dotted) and 
						real part of the approximation by $50000$ (2 ``micro'' steps for each ``macro'' time step, this means $k = 2$) 
						time steps of the complex explicit \textsc{Euler} method (dashed), induced by $\Upsilon_1$.}
	\label{fig:arenstorfOrbit}
\end{sidewaysfigure}
\nocite{F01}
\nocite{F02}
\nocite{FIL06}
\section{Outlook}
We have seen, that composition methods with complex coefficients is nothing else than numerical integration along special complex time grids
using a basic (RKM). As mentioned in the introduction, numerical integration along complex time grids involves several problems
in general. 

First, one has to handle monodromy effects. This means, if one uses the discussed complex methods, one has to be aware of
the fact, that branch points close to the real axis can be circulated by the complex path of integration, which causes a switch
of the calculated solution's branch.  

Moreover, one has to develop a computing framework, 
that enables the integrator to mimic the concept of global analytic functions. 

That all this effort is a worthwhile endeavor can be seen as follows. On the hand, the algorithmic effort of numerically solving an initial 
value problem by taking a complex detour, can be reduced in a significant way (compare \cite{C80}). Thereby the idea is to stay away from
singularities. On the other hand, the extended viewpoint given by studying differential equations over the complex numbers enriches 
the structural understanding of the solution. For example, it might be possible and of practical interest to solve boundary value problems
over the complex numbers, in situations where a real-valued approach would be ill-posed.

\section*{Acknowledgement}
The authors wish to express their gratitude to Folkmar Bornemann, Juri Suris, Oliver Junge and Ernst Hairer for their helpful remarks.
Furthermore we thank Gilles Vilmart for his support.
\bibliography{goncti}

\begin{thebibliography}{10}

\bibitem{A63}
{\sc R.~Arenstorf}, {\em {Periodic solutions of the restricted three body
  problem representing analytic continuations of Keplerian elliptic motions}},
  American Journal of Mathematics, 85 (1963), pp.~27--35.

\bibitem{BORNE}
{\sc F.~Bornemann, D.~Laurie, S.~Wagon, and J.~Waldvogel}, {\em The SIAM
  100-Digit Challenge: A Study in High-Accuracy Numerical Computing}, Society
  of Industrial Applied Mathematics (SIAM), Philadelphia, 2004.

\bibitem{CGSS05}
{\sc F.~Calogero, D.~Gmoez-Ullate, P.~Santini, and M.~Sommacal}, {\em {The
  transition from regular to irregular motion explained as travel on Riemann
  surfaces}}, Journal of Physics A: Mathematical and General, 38 (1980),
  pp.~8873--8896.

\bibitem{C80}
{\sc G.~Corliss}, {\em Integrating ode's in the complex plane - pole vaulting},
  Mathemtics of Computation, 35 (1980), pp.~1181--1189.

\bibitem{FIL06}
{\sc Denis~M. Filatov}, {\em {On complex-stepped Runge-Kutta methods for exact
  time integration of linear PDEs}}, HAIT Journal of Science and Engineering C,
  4 (2006), pp.~128--135.

\bibitem{F01}
{\sc T.C. Fung}, {\em {Third order complex-time-step methods for transient
  analysis}}, Computer methods in applied mechanics and engineering, 190
  (2001), pp.~2789--2802.

\bibitem{F02}
\leavevmode\vrule height 2pt depth -1.6pt width 23pt, {\em {Solving non-linear
  problems by complex time step methods}}, Communications in Numerical Methods
  in Engineering, 18 (2002), pp.~287--303.

\bibitem{HLW02}
{\sc E.~Hairer, C.~Lubich, and G.~Wanner}, {\em Geometric Numerical
  Integration. Structure-Prevserving Algorithms for Ordinary Differential
  Equations}, no.~31 in Springer Series in Computational Mathematics,
  Springer-Verlag, Berlin, first~ed., 2002.

\bibitem{HO08}
{\sc E.~Hansen and A.~Ostermann}, {\em Higher order splitting methods for
  analytic semigroups exists}, Submitted,  (2008).

\bibitem{KRG1}
{\sc U.~Kortenkamp and J.~Richter-Gebert}, {\em Complexity issues in dynamic
  geometry}, in {Festschrift in the honor of Stephen Smale's 70th birthday},
  M.~Rojas and F.~Cucker, eds., World Scientific, 2002, pp.~355--404.

\bibitem{KRG2}
{\sc J.~Richter-Gebert and U.~Kortenkamp}, {\em Cinderella -- the interactive
  geometry software}, Springer-Verlag, Berlin, Heidelberg, New York, 2000.

\bibitem{SS73}
{\sc F.~Sch{\"a}fke and D.~Schmidt}, {\em Gew{\"o}hnliche
  Differentialgleichungen - Die Grundlagen der Theorie im Reellen und
  Komplexen}, Springer-Verlag, Berlin, Heidelberg, New York, 1973.

\bibitem{V08}
{\sc G.~Vilmart}, {\em {\'E}tude d'int{\'e}grateurs g{\'e}om{\'e}triques pour
  des {\'e}quations diff{\'e}rentielles}, PhD thesis, Universit{\'e} de
  Gen{\`e}ve, 2008.

\end{thebibliography}
\bibliographystyle{siam}
\end{document}